%% file: draft-22july.tex
\documentclass[hidelinks,onefignum,onetabnum]{siamart250211}

\input{ex_shared}

\ifpdf
\hypersetup{
  pdftitle={An Example Article},
  pdfauthor={D. Doe, P. T. Frank, and J. E. Smith}
}
\fi

\begin{document}

\maketitle

% REQUIRED
\begin{abstract}
Transformers are effective and efficient at modeling complex relationships and learning patterns from structured data in many applications.
The main aim of this paper is to propose and design NLAFormer, which is
a transformer-based architecture for learning numerical linear algebra operations: pointwise computation, shifting, transposition, inner product, matrix multiplication, and matrix-vector multiplication. 
Using a linear algebra argument, we demonstrate that transformers can express such operations. 
Moreover, the proposed approach discards the simulation of computer control flow adopted by the loop-transformer, significantly reducing both the input matrix size and the number of required layers. By assembling linear algebra operations, NLAFormer can learn the conjugate gradient method to solve symmetric positive definite linear systems. Experiments are conducted to illustrate the numerical performance of NLAFormer.
\end{abstract}

% REQUIRED
\begin{keywords}
Transformers, numerical linear algebra, machine learning
\end{keywords}

% REQUIRED
\begin{MSCcodes}
65F10, 68Q32
\end{MSCcodes}

\section{Introduction}

The transformer model, introduced by Vaswani et al. in 2017, was originally developed for natural language processing (NLP) tasks \cite{vaswani2017attention}. 
By combining multi-head self-attention with feedforward neural networks (FFNs), transformers effectively capture global dependencies in input sequences. 
Transformers have proven to be effective in modeling complex relationships and learning patterns from structured data \cite{diao2022relational}. Their success has revolutionized fields such as natural language processing, computer vision, and scientific reasoning, see \cite{choi2019graph,islam2024comprehensive}. \begin{figure}[ht]
    \centering
    \includegraphics[width=0.8\linewidth]{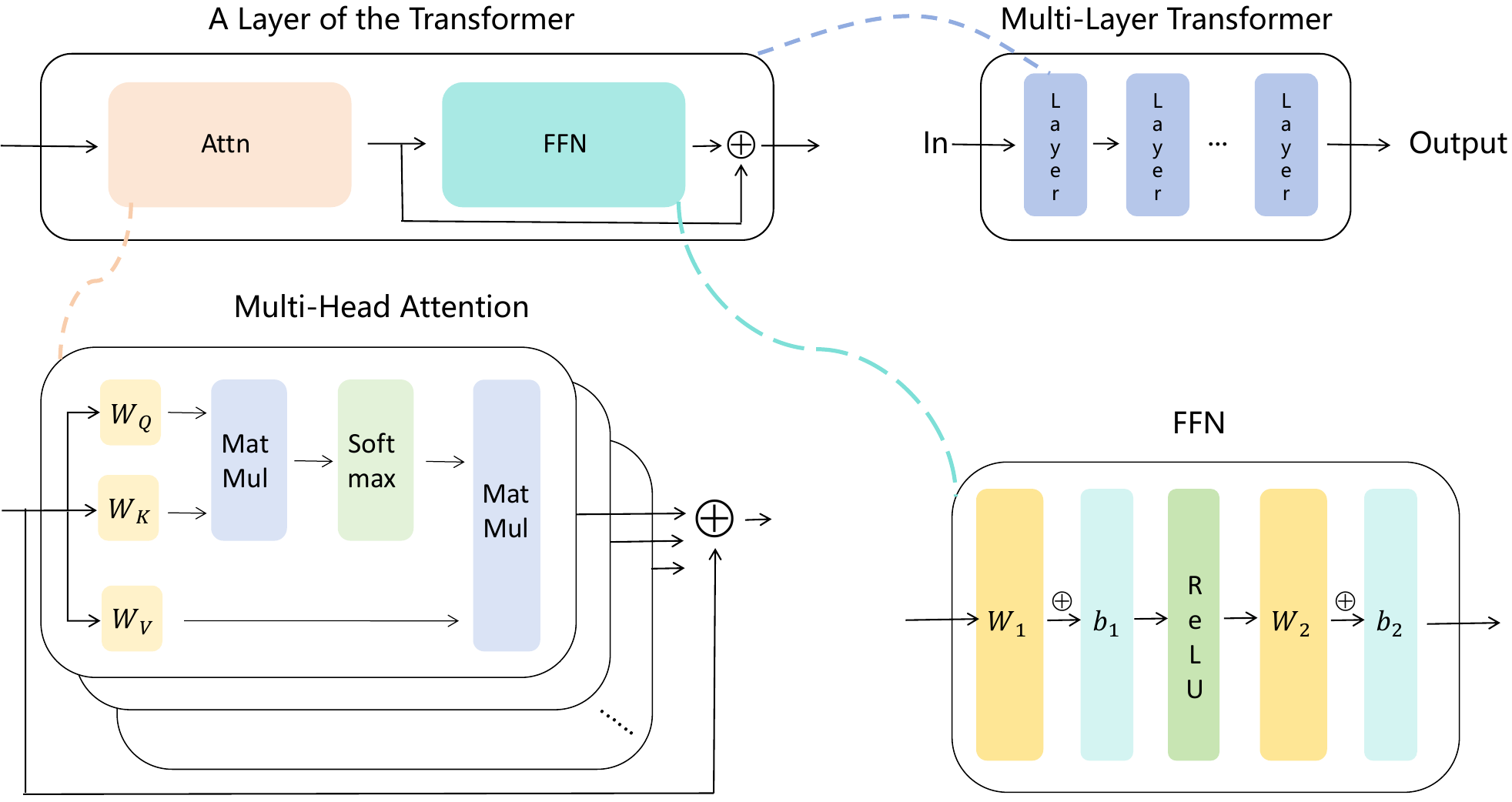}
    \caption{The components of a multi-layer transformer}
    \label{Transformer diagram}
\end{figure}

In \cref{Transformer diagram}, we illustrate the components of a multi-layer transformer.
Each single-layer transformer consists of a multi-head attention module and a feedforward neural network.
The attention module contains three elements: queries (Q), keys (K), and values (V).
These three elements are represented by three matrices ${\bf W}_Q$, ${\bf W}_K$ and ${\bf W}_V$ respectively. 
For an input token ${\bf P}$, the attention mechanism computes similarity scores between queries and keys to determine how much attention each position should pay to others. These scores are normalized using softmax and used to calculate a weighted sum of the values, allowing each position to incorporate context from the entire sequence and efficiently model global dependencies.
The output $\text{TF}(\mathbf{P})$ is given by the following computational procedure:
\begin{equation}
\label{DefTrans}
    \begin{aligned}
& \text{Attn}(\mathbf{P}) = \mathbf{P} + \sum_{i=1}^{h} \mathbf{W}_V^{(i)} \mathbf{P} \cdot \text{softmax} \left( \mathbf{P}^\top \mathbf{W}_K^{(i)\top} \mathbf{W}_Q^{(i)} \mathbf{P} \right), \\
& \text{FFN}(\mathbf{P})=\mathbf{W}_2 \text{ReLU}(\mathbf{W}_1 \mathbf{P} + \mathbf{b}_1 \mathbf{1}^\top) + \mathbf{b}_2\mathbf{1}^\top,\\
& \text{TF}(\mathbf{P}) = \text{Attn}(\mathbf{P}) + \text{FFN}(\text{Attn}(\mathbf{P})),\\
& \text{where} \ \text{softmax}(\mathbf{Z})_{i,j} = \frac{e^{{Z}_{i,j}}}{\sum_{k=1}^{d} e^{{Z}_{k,j}}}, \quad
\text{ReLU}(\mathbf{Z})_{i,j} = \max(0, {{Z}_{i,j}}), \\
& \quad \quad \quad \mathbf{1} \ \text{is a vector of all ones}. \\
\end{aligned}
\end{equation}

In this process, multi-head attention is used to capture different types of relationships. The output is then passed through a feedforward network, which enhances the ability of the model to learn complex mappings using two linear transformations ($\mathbf{W}_1, \mathbf{W}_2$)
with biased vectors ($\mathbf{b}_1, \mathbf{b}_2$), and a nonlinear rectified linear unit (ReLU). 
A multi-layer transformer stacks several single-layer transformers sequentially, with each layer refining the output of the previous one to capture increasingly complex patterns in the input. %This architecture addresses the limitations of recurrent neural networks, particularly their difficulty in modeling long-range dependencies and their inefficiency in sequential processing.

\subsection{Numerical Linear Algebra Operations}

Numerical linear algebra supports a wide range of applications, including scientific computing, engineering simulations, and economic modeling. It serves as a foundational tool for precise and efficient problem solving in these fields \cite{jbilou2021numerical}. At its core, numerical linear algebra comprises fundamental mathematical operators such as shifting, matrix transposition, multiplication, inner products, etc. These operators are not merely computational tools; they form the theoretical basis of classical iterative numerical algorithms \cite{trefethen2022numerical}. 

A previous study \cite{Charton2021LinearAW} experimentally verified that the transformer can perform basic numerical linear algebra operations by adopting a very specialized encoding. In \cite{Giannou2023LoopedTA}, Giannou et al.
developed a loop-transformer model as a programmable computer. 
More precisely, they designed and mapped transformer
models to fundamental computer modules such as scratchpad, memory, and instruction.
The loop-transformer model can be employed to simulate the execution of basic numerical linear algebra operations. 
In \cite{Yang2023LoopedTA}, Yang et al. designed the looped transformers to learn algorithms such as sparse linear regression. In \cite{gao2024expressive}, Gao et al. incorporated pre- and post-transformer modules, inspired by the common pre-processing and post-processing steps in traditional algorithms, to enhance performance in complex learning tasks. 

Despite these promising developments, existing approaches that simulate computer control flow with transformers remain inefficient and ill-suited for numerical linear algebra operations. The fundamental issue lies in the attempt to emulate each primitive operation, such as memory access, data movement, and basic arithmetic, using a whole transformer block. For example, transposing an $n \times n$ matrix using the loop-transformer \cite{Giannou2023LoopedTA} requires $O(n^3)$ input and a 4-layer transformer architecture to ensure successful learning. Moreover, for matrix multiplication between $n \times n$ matrices: $\mathbf{A}^{\top}\mathbf{B}$ , the loop-transformer requires $O(n^3)$ input and a 6-layer transformer architecture.
These observations motivate a deeper exploration of transformers in the domain of numerical linear algebra by distilling the transformer architecture, identifying the components that are genuinely useful for numerical linear algebra while eliminating those that introduce unnecessary computations and storage.

\subsection{Our Contributions}

In this paper, we propose and design Numerical Linear Algebra transFormer (NLAFormer), which is 
a transformer-based architecture for learning numerical linear algebra operations.
Using a linear algebra argument, we show that the proposed NLAFormer
has the capacity to express numerical linear algebra operations such as
pointwise computation, shifting, transposition, inner product, matrix multiplication, and matrix-vector multiplication.
Moreover, our approach discards the simulation of computer control flow adopted by loop-transformer, significantly reducing both the size of the input matrix and the required number of layers. For example, in the case of $n$-by-$n$ matrix transposition, our approach only requires an input of size $O(n^2)$ and a 2-layer transformer; see the results in Section 2.
Similar reductions apply to other basic numerical linear algebra operations. 
Our NLAFormer can focus on learning meaningful representations, making it highly effective for general-purpose numerical linear algebra solvers.
By assembling linear algebra operations, we demonstrate that NLAFormer can learn the conjugate gradient method for solving symmetric positive definite linear systems. Experimental results are conducted to illustrate the numerical performance of the proposed NLAFormer, showing that it can internalize the iterative logic of classical solvers and, through data-driven training, identify update strategies for faster convergence.

The outline of this paper is given as follows. 
In Section 2, we develop and design transformers for numerical linear algebra operations and show their expressiveness for such operations. 
In Section 3, we demonstrate how NLAFormer can learn a conjugate gradient method to solve symmetric positive definite linear systems.
In Section 4, we present numerical results to illustrate the numerical performance of the proposed NLAFormer. In Section 5, we provide concluding remarks and present future work.
% \cite{paul2022vision, ziemann2022transformers, kusters2023cnns, aspillaga2020stress, karita2019comparative,brown2020language, devlin2018bert}.
\subsection{Notations}

We use boldface capital and lowercase letters to denote matrices and vectors, respectively. Non-bold letters represent the elements of matrices, vectors, or scalars. The subscripts represent the serial numbers or the elements of the specified rows and columns of the matrix. For example, $\mathbf{P}_{r}$ means the $r$th row of $\mathbf{P}$, ${P}_{i,j}$ means the $i$-th row, $j$-th column of $\mathbf{P}$. Calligraphic letters are used to represent sets.
% To make the notation less complicated, even if we mark several rows of a matrix, we still use lowercase subscripts instead of bold subscripts.

\section{Learning Numerical Linear Algebra Operations}
\label{sec:Learning Numerical Linear Algebra Operations}
A central step toward understanding the expressive capacity of transformers in numerical linear algebra is to investigate whether they can represent the fundamental operators that form the basis of classical numerical methods. Regardless of the complexity of a numerical algorithm, its computational structure ultimately relies on a small set of core operations, including pointwise computation, shifting, transposition, inner product, matrix multiplication, matrix-vector multiplication, etc. We compare the different requirements of NLAFormer and loop-transformer in handling basic operators in \cref{Tablecopmpare}. As both methods require a few heads, we only present their required input matrix sizes and number of layers in the table.

\begin{table}
\centering
\begin{tabular}{|c|l|l|l|}
\hline
Operation & Requirement   & NLAFormer                & Loop-Transformer      \\ \hline
                              \multirow{2}{*}{\begin{tabular}[c]{@{}c@{}}Pointwise $+, -, \cdot, \div$ \\ between vectors\end{tabular}} &
  Input size &
  $O(n)$ &
  $O(n)$ \\ \cline{2-4} 
                              & Number of layers   & $1$                  & $12n$                  \\ \hline
 \multirow{2}{*}{Column shift}         & Input size & $O(n)$                    & $O(n)$                    \\ \cline{2-4} 
                              & Number of layers   & $1$                & $4$                 \\ \hline                             
\multirow{2}{*}{Row shift}         & Input size & $O(n)$                    & $O(n \log(n))$             \\ \cline{2-4} 
                              & Number of layers   & $1$                & $4n$                  \\ \hline
\multirow{2}{*}{Vector transpose}         & Input size & $O(n^2)$ & $O(n^3)$ \\ \cline{2-4} 
                              & Number of layers   & $1$                   & $4$                  \\ \hline
\multirow{2}{*}{$\mathbf{a}^{\top}\mathbf{b}$}         & Input size & $O(n^2)$  & $O(n^3)$ \\ \cline{2-4} 
                              & Number of layers   & $1$                   & $6$                  \\ \hline
\multirow{2}{*}{$\mathbf{a}\mathbf{b}^{\top}$}         & Input size & $O(n^2)$ & $O(n^3)$ \\ \cline{2-4} 
                              & Number of layers   & $1$                   & $6$                  \\ \hline                              
\multirow{2}{*}{Matrix transpose}         & Input size & $O(n^2)$ & $O(n^3)$ \\ \cline{2-4} 
                              & Number of layers   & $1$                 & $4$                 \\ \hline

\multirow{2}{*}{$\mathbf{A}^{\top}\mathbf{B}$}         & Input size & $O(n^2)$ & $O(n^3)$ \\ \cline{2-4} 
                              & Number of layers   & $1$                   & $6$                  \\ \hline
\multirow{2}{*}{$\mathbf{A}\mathbf{B}$}         & Input size & $O(n^2)$ & $O(n^3)$ \\ \cline{2-4} 
                              & Number of layers   & $2$               & $10$                 \\ \hline

\multirow{2}{*}{$\mathbf{A}\mathbf{b}$} & Input size & $O(n^2)$ & $O(n^3)$ \\ \cline{2-4} 
                              & Number of layers   & $2$                  & $6$                   \\ \hline                              
\end{tabular}
\caption{Different requirements between NLAFormer and loop-transformer.}
\label{Tablecopmpare}
\end{table}

Here we present \cref{TheAlgformerTheorembasicopr} to show the expressiveness of transformers for handling ten basic numerical linear algebra operations in terms of the number of layers and heads in the structure, where each layer consists of a self-attention mechanism followed by an FFN, and each head represents an independent attention mechanism that captures different structural aspects of the input, as illustrated in \cref{Transformer diagram}. Throughout the following proof, it is assumed that all vectors \( \mathbf{a} \in \mathbb{R}^n \) are in a compact subset \( \mathcal{K}_{n} \subseteq \mathbb{R}^n \), and that all matrices \( \mathbf{A} \in \mathbb{R}^{n \times n} \) are assumed to lie in a compact subset \( \mathcal{K}_{n \times n} \subseteq \mathbb{R}^{n \times n} \). Unless specified, norm $\|\cdot\|$ is taken to be the Frobenius norm for matrices.

\begin{theorem}
\label{TheAlgformerTheorembasicopr}

\begin{enumerate}
\item[(i)] For any $\mathbf{a},\mathbf{b}\in \mathbb{R}^{n}$, there exists a one-layer, one-head transformer to execute the point-wise addition, subtraction, multiplication, and division (with $b_{i}\neq 0$) between vectors: 
\begin{equation}
    \mathbf{P}=\begin{bmatrix}
        {a}_{1} &\dots &  {a}_{n}\\
        {b}_{1} &\dots &  {b}_{n}\\
        0 &\dots &  0
    \end{bmatrix} \quad {\rm and} \quad\text{TF}(\mathbf{P})\approx\begin{bmatrix}
        {a}_{1} &\dots &  {a}_{n}\\
        {b}_{1} &\dots &  {b}_{n}\\
        {d}_{1} &\dots &  {d}_{n}
    \end{bmatrix},
\end{equation}
where $\mathbf{d}=\mathbf{a} \square \mathbf{b}$, with $\square \in \{+, -, \cdot, \div\}$, and $a_{i},b_{i},d_{i},$ for $1\leq i\leq n$ denote the $i$th elements of the respective vectors;
\item[(ii)]    For any $\mathbf{a},\mathbf{b}\in \mathbb{R}^{n}$, there exists a one-layer, two-head transformer to execute the column shifting:
    \begin{equation}
       \mathbf{P}=\begin{bmatrix}
        \bf{0}&\mathbf{a}&\mathbf{b}\\
 \bf{0}&\bf{0}&\bf{0}\\
\mathbf{e}_1&\mathbf{e}_{2}&\mathbf{e}_{3}
    \end{bmatrix} \quad {\rm and} \quad \text{TF}(\mathbf{P})\approx\begin{bmatrix}
        \bf{0}&\mathbf{b}&\mathbf{a}\\
 \bf{0}&\bf{0}&\bf{0}\\
        \mathbf{e}_{1} &\mathbf{e}_{2}&\mathbf{e}_{3}
    \end{bmatrix},
    \end{equation}
where $\mathbf{e}_{1}$, $\mathbf{e}_{2}$ and $\mathbf{e}_{3}$ are the first, second and third unit vectors 
of the 3-by-3 identity matrix; there exists a one-layer, one-head transformer to execute the row shifting:
    \begin{equation}\mathbf{P}=\begin{bmatrix}
        \mathbf{a}^{\top}\\
\mathbf{b}^{\top}
    \end{bmatrix} \quad and \quad \text{TF}(\mathbf{P})\approx\begin{bmatrix}
        \mathbf{b}^{\top}\\
\mathbf{a}^{\top}
    \end{bmatrix};    
   \end{equation}
there exists a one-layer, two-head transformer to execute the transpose of $\mathbf{a}$ 
such that  
\begin{equation}
        \mathbf{P}=
        \begin{bmatrix}
            {0}& &   \mathbf{a}^{\top} & \\
            \mathbf{0}& \bf{0} & \dots & \mathbf{0}\\
            \mathbf{e}_{1} & \mathbf{e}_2 & \dots &\mathbf{e}_{n+1}
        \end{bmatrix} \quad {\rm and} \quad \text{TF}(\mathbf{P})\approx
         \begin{bmatrix}
            \mathbf{0}& \dots & \bf{0} &\mathbf{a}\\
            {0}& \dots & 0 & {0}\\
            \mathbf{e}_{1}& \dots & \mathbf{e}_n & \mathbf{e}_{n+1}
        \end{bmatrix},
    \end{equation}
    where $\mathbf{e}_{1}, \dots, \mathbf{e}_{n+1}$ are the first to $(n{+}1)$th unit vectors of the $(n{+}1) \times (n{+}1)$ identity matrix;
    \item[(iii)]
    For any $\mathbf{a},\mathbf{b}\in \mathbb{R}^{n}$, there exists a one-layer, two-head transformer to execute the inner product $\mathbf{a}^{\top}\mathbf{b}$:
    \begin{equation}
    \mathbf{P}=\begin{bmatrix}
        {0}&&\mathbf{a}^{\top} & \\
        {0}&&\mathbf{b}^{\top} & \\
        {0}& 0 & \dots & {0}\\
        \mathbf{e}_{1}& \mathbf{e}_2 & \dots& \mathbf{e}_{n+1}
    \end{bmatrix} \quad {\rm and} \quad  \text{TF}(\mathbf{P})\approx   \begin{bmatrix}
        {0}& &\mathbf{a}^{\top} & \\
        {0}&&\mathbf{b}^{\top}& \\
        {0}& 0 & \dots & \mathbf{a}^{\top}\mathbf{b}\\
        \mathbf{e}_{1}& \mathbf{e}_2 & \dots&  \mathbf{e}_{n+1}
    \end{bmatrix};
\end{equation}
    there exists a one-layer, two-head transformer to execute the product between vectors
    $\mathbf{a}\mathbf{b}^{\top}$:
\begin{equation}
    \mathbf{P}=\begin{bmatrix}
        {0}&&\mathbf{a}^{\top} & \\
        {0}&&\mathbf{b}^{\top} & \\
        \mathbf{0}& \mathbf{0} & \dots & \mathbf{0}\\
        \mathbf{e}_{1}&\mathbf{e}_2 & \dots & \mathbf{e}_{n+1}
    \end{bmatrix} \quad {\rm and} \quad  \text{TF}(\mathbf{P})\approx   \begin{bmatrix}
        {0}&&\mathbf{a}^{\top} & \\
        {0}&&\mathbf{b}^{\top} & \\
        \mathbf{0}&&\mathbf{a}\mathbf{b}^{\top} & \\
        \mathbf{e}_{1}&\mathbf{e}_2 & \dots & \mathbf{e}_{n+1}
    \end{bmatrix}; 
\end{equation}
\item[(iv)]
For any $\mathbf{A},\mathbf{B}\in\mathbb{R}^{n\times n}, \mathbf{b}\in \mathbb{R}^{n}$, there exists a one-layer, two-head transformer to execute the transpose of a matrix such that 
   \begin{equation}
    \mathbf{P}=\begin{bmatrix}
        \mathbf{0}& & \mathbf{A} & \\
         \mathbf{0}& \bf{0} & \cdots & \bf{0} \\
        \mathbf{e}_{1}& \mathbf{e}_2 & \dots &\mathbf{e}_{n+1}
    \end{bmatrix} \quad {\rm and} \quad \text{TF}(\mathbf{P})\approx\begin{bmatrix}
        \mathbf{0}& &\mathbf{A}^{\top} & \\
         \mathbf{0}& \bf{0} & \dots & \mathbf{0}\\
        \mathbf{e}_{1} & \mathbf{e}_2 &\dots&\mathbf{e}_{n+1}
    \end{bmatrix};
\end{equation}
there exists a one-layer, two-head transformer to execute the multiplication between matrices: $\mathbf{A}^{\top}\mathbf{B}$ such that 
  \begin{equation}
    \mathbf{P}=\begin{bmatrix}
        \mathbf{0}& &\mathbf{A} & \\
        \mathbf{0}& &\mathbf{B} & \\
        \mathbf{0}& \mathbf{0} & \dots &\mathbf{0}\\
        \mathbf{e}_{1}& \mathbf{e}_2 & \dots&\mathbf{e}_{n+1}
    \end{bmatrix} \quad {\rm and} \quad  \text{TF}(\mathbf{P})\approx   \begin{bmatrix}
     \mathbf{0}& &\mathbf{A} & \\
        \mathbf{0}& &\mathbf{B} & \\
        \mathbf{0}&  & \mathbf{A}^{\top}\mathbf{B}  & \\
        \mathbf{e}_{1}& \mathbf{e}_2 & \dots&\mathbf{e}_{n+1}
    \end{bmatrix}; 
\end{equation}
there exists a two-layer, two-head transformer to execute the multiplication between matrices: $\mathbf{A}
\mathbf{B}$ such that 
  \begin{equation}
    \mathbf{P}=\begin{bmatrix}
        \mathbf{0}& &\mathbf{A} & \\
        \mathbf{0}& &\mathbf{B} & \\
        \mathbf{0}& \mathbf{0} & \dots &\mathbf{0}\\
        \mathbf{e}_{1}& \mathbf{e}_2 & \dots&\mathbf{e}_{n+1}
    \end{bmatrix} \quad {\rm and} \quad  \text{TF}(\mathbf{P})\approx   \begin{bmatrix}
     \mathbf{0}& &\mathbf{A} & \\
        \mathbf{0}& &\mathbf{B} & \\
        \mathbf{0}&  & \mathbf{A} \mathbf{B}  & \\
        \mathbf{e}_{1}& \mathbf{e}_2 & \dots&\mathbf{e}_{n+1}
    \end{bmatrix};
\end{equation}
there exists a two-layer, two-head transformer to execute the multiplication between matrix and vector: $\mathbf{A}
\mathbf{b}$ such that 
  \begin{equation}
    \mathbf{P}=\begin{bmatrix}
        \mathbf{0}& &\mathbf{A} & \\
        {0}& &\mathbf{b}^{\top} & \\
        \mathbf{0}& \mathbf{0} & \dots &\mathbf{0}\\
        \mathbf{e}_{1}& \mathbf{e}_2 & \dots&\mathbf{e}_{n+1}
    \end{bmatrix} \quad {\rm and} \quad  \text{TF}(\mathbf{P})\approx   \begin{bmatrix}
     \mathbf{0}& &\mathbf{A} & \\
        {0}& &\mathbf{b}^{\top} & \\
        \mathbf{0}&  & \mathbf{A} \mathbf{b}  & \\
        \mathbf{e}_{1}& \mathbf{e}_2 & \dots&\mathbf{e}_{n+1}
    \end{bmatrix}.
\end{equation}
\end{enumerate}
\end{theorem}

\begin{proof}
%[Proof of \cref{TheAlgformerTheorembasicopr}]
%      For (i), the input prompt is depicted as follows:
%     $$ \mathbf{P}=\begin{bmatrix}
%         {a}_{1} &\dots &  {a}_{n}\\
%         {b}_{1} &\dots &  {b}_{n}\\
%          {0}&\dots &  {0}
%     \end{bmatrix} .    
%     $$
% We take $\mathbf{W}_{Q,K,V}=0$ to have Attn$(\textbf{P})=\textbf{P}$. By the universal approximation theorem \citep[Theorem 2]{leshno1993multilayer}, for any continuous function $f(x,y)$ and $\epsilon>0$, there exists an FFN, such that $\sup_{x,y}|f(x,y)-\text{FFN}(x,y)|<\epsilon$. Therefore, the FFN can approximate the following operations:
% $$
% \mathbf{P}_{1,j} \square \mathbf{P}_{2,j}=\mathbf{d}_{j},\quad \square \in \{+, -, \cdot, \div\},
% $$ for $1\leq j\leq n$.
% Hence, we have the desired result:
%  \begin{equation}
%       \text{TF}(\mathbf{P})=        \begin{bmatrix}
%         {a}_{1} &\dots &  {a}_{n}\\
%         {b}_{1} &\dots &  {b}_{n}\\
%         {d}_{1} &\dots &  {d}_{n}
%     \end{bmatrix}.
%  \end{equation}
%  Note: If $d_i$ is replaced by a function $f(a_i, b_i)$, where $f$ is any composition of addition, subtraction, multiplication, or division, the result still holds. 
For (i), we consider the input prompt matrix:
\begin{equation}
\mathbf{P} = \begin{bmatrix}
    a_1 & \dots & a_n \\
    b_1 & \dots & b_n \\
    0 & \dots & 0
\end{bmatrix}.
\end{equation}
We take all attention parameters to be zero, i.e., $\mathbf{W}_{Q,K,V} = 0$, so that the attention module simply returns the input unchanged:
\begin{equation}
    \mathrm{Attn}(\mathbf{P}) = \mathbf{P}.
\end{equation}
We now consider how the FFN processes this matrix. Since FFN acts independently and same on each column, it suffices to analyze a single column of $\mathbf{P}$. For simplicity, we write FFN as acting on a column vector below, with the understanding that the same operation is applied to every column. 
By using the universal approximation theorem \cite[Theorem 2]{leshno1993multilayer}, for continuous function $\textbf{f} \in C(\mathcal{K}, \mathbb{R}^3)$ on a compact set $\mathcal{K} \subseteq \mathbb{R}^3$ and any $\varepsilon > 0$, there exists an FFN such that:
\begin{equation}
    \sup_{\mathbf{u} \in \mathcal{K}} \left\| \mathrm{FFN}(\mathbf{u}) - \mathbf{f}(\mathbf{u}) \right\| < \frac{\varepsilon}{n}, \quad \mathbf{f}(\mathbf{u})_r = 
\begin{cases}
    \mathbf{u}_1 \mathbin{\square} \mathbf{u}_2, & \text{if } r = 3 \\
    0, & \text{otherwise}
\end{cases},
\end{equation}
 where $\square \in \{+, -, \cdot, \div\}$ (with denominator $\neq 0$ when dividing). Therefore, for each column of $\mathbf{P}$, the FFN can be constructed to approximate the result of $a_j \square b_j$ and write it into the third row. Applying this column-wise to $\mathbf{P}$, we obtain:
\begin{equation}
    \left\| \mathrm{FFN}(\mathrm{Attn}(\mathbf{P})) - \mathbf{T} \right\| < \varepsilon, \quad \text{where } 
    \mathbf{T} = \begin{bmatrix}
        0 & \dots & 0 \\
        0 & \dots & 0 \\
        d_1 & \dots & d_n
    \end{bmatrix}, \quad d_j = a_j \square b_j.
\end{equation}
Thus, we obtain the final result, where the approximation comes from the FFN:
\begin{equation}
    \mathrm{TF}(\mathbf{P}) = \mathrm{Attn}(\mathbf{P}) + \mathrm{FFN}(\mathrm{Attn}(\mathbf{P})) \approx \mathbf{P} + \mathbf{T} = 
    \begin{bmatrix}
        a_1 & \dots & a_n \\
        b_1 & \dots & b_n \\
        d_1 & \dots & d_n
    \end{bmatrix},
\end{equation}
Remark that if $d_i$ is replaced by a function $f(a_i, b_i)$, where $f$ is any composition of addition, subtraction, multiplication, or division, the result still holds.

For (ii), the input prompt is depicted as follows:
    $$ \mathbf{P}=\begin{bmatrix}
        \textbf{0}&\mathbf{a}&\mathbf{b}\\
 \textbf{0}&\textbf{0}&\textbf{0}\\
\mathbf{e}_1&\mathbf{e}_{2}&\mathbf{e}_{3}
    \end{bmatrix} .    $$
We construct 
$$\mathbf{W}_{Q}=\begin{bmatrix}
    \mathbf{0}& C&C&C\\
    \mathbf{0}&0&0&c\\
    \mathbf{0}&0&c&0
\end{bmatrix}, \ \mathbf{W}_{Q}\mathbf{P}=\begin{bmatrix}
 C&C&C\\
0&0&c\\
0&c&0
\end{bmatrix}, \ 
\mathbf{W}_{K}=\begin{bmatrix}
    \mathbf{0}&\mathbf{E}_{3}
\end{bmatrix}, \ \mathbf{W}_{K}\mathbf{P}=
\begin{bmatrix}
\mathbf{E}_{3}
\end{bmatrix},
$$
$$
\mathbf{W}_{V}=e^{C}\begin{bmatrix}
    \mathbf{0}&\mathbf{0}\\
    \mathbf{E}_{n}&\mathbf{0}\\
    \mathbf{0}&\mathbf{0}
\end{bmatrix}, \quad \mathbf{W}_{V}
\mathbf{P}=e^{C}\begin{bmatrix}
 \textbf{0}&\textbf{0}&\textbf{0}\\
        \textbf{0}&\mathbf{a}&\mathbf{b}\\
\mathbf{0}&\mathbf{0}&\mathbf{0}
    \end{bmatrix},
$$
where $C,\ c$ are positive constants, and the $n$-by-$n$ 
identity matrix $\mathbf{E}_{n}$. Then we mark:
$$
\mathbf{Z}=\mathbf{P}^{T}\mathbf{W}^{T}_{K}\mathbf{W}_{Q}\mathbf{P}=
\begin{bmatrix}
 C&C&C\\
0&0&c\\
0&c&0
\end{bmatrix}.
$$
By using similar deduction in \cite{gao2024expressive,Giannou2023LoopedTA}, we know the following equation holds for $2\leq i,j\leq 3$:
\begin{equation}
    (e^{C}\text{softmax}(\mathbf{Z}))_{i,j}\approx1+{Z}_{i,j},
\end{equation}
where the approximation can be arbitrarily well, as $C$ is sufficiently large and $c$ is sufficiently small. After we introduce another head to cancel the constant (see
\cite{gao2024expressive}), we obtain that the output of the attention layer: 
\begin{equation}
    \sum_{h=1}^{2}\mathbf{W}_{V}^{(h)}\mathbf{P} \cdot \text{softmax}(\mathbf{P}^{T}\mathbf{W}^{(h)T}_{K}\mathbf{W}_{Q}^{(h)}\mathbf{P})\approx c\begin{bmatrix}
 \textbf{0}&\textbf{0}&\textbf{0}\\
        \textbf{0}&\mathbf{b}&\mathbf{a}\\
\mathbf{0}&\mathbf{0}&\mathbf{0}
    \end{bmatrix}.
\end{equation}
Therefore, we have  
\begin{equation}
    \text{Attn}(\mathbf{P})\approx\mathbf{U}=\begin{bmatrix}
        \textbf{0}&\mathbf{a}&\mathbf{b}\\
 \textbf{0}&c\textbf{b}&c\textbf{a}\\
\mathbf{e}_1&\mathbf{e}_{2}&\mathbf{e}_{3}
    \end{bmatrix}.
\end{equation}

%The following example illustrates how FFN works. This serves as a representative case; the remaining FFN-related parts in the %theorem can be established in a similar manner. 
First, let $r_{1}$ denotes the row containing $\mathbf{a}$ and $\mathbf{b}$, 
$r_{2}$ denotes the row containing $c\mathbf{b}$ and $c\mathbf{a}$ in $\mathbf{U}$. 
Although the FFN takes the matrix as input, it processes each column independently and identically. Thus, it suffices to prove the approximation for a single column; the same applies to all others. For simplicity, we treat FFN as acting on a column vector below, with the understanding that the same operation is applied to every column. By using the universal approximation theorem \cite[Theorem 2]{leshno1993multilayer}, for all $\varepsilon>0$, compact $\mathcal{K}\subseteq \mathbb{R}^{2n+3}$, the FFN can approximate the continuous function $\mathbf{f}\in C(\mathcal{K},\mathbb{R}^{2n+3})$:
\begin{equation}
\label{elimate}
\sup_{\textbf{u} \in \textbf{K}} \left\| \text{FFN}(\textbf{u}) - \textbf{f}(\mathbf{u}) \right\| < \frac{\varepsilon}{3},\quad  \mathbf{f}(\mathbf{u})_r = 
\begin{cases}
    -\mathbf{u}_{r_{1}} + \frac{1}{c} \mathbf{u}_{r_{2}}, & \text{if } r = r_1 \\
    -\mathbf{u}_{r_{2}}, & \text{if } r = r_2 \\
    \mathbf{0}, & \text{otherwise}
\end{cases}.
\end{equation}
Therefore, by applying $\mathbf{f}$ column-wise to $\mathbf{U}$, we obtain:
\begin{equation}
    \left\| \mathrm{FFN}(\mathbf{U}) - \mathbf{T} \right\| < \varepsilon, \quad \text{where } 
    \mathbf{T} = \begin{bmatrix}
        \mathbf{0} & -\mathbf{a} + \mathbf{b} & -\mathbf{b} + \mathbf{a} \\
        \mathbf{0} & -c\mathbf{b} & -c\mathbf{a} \\
        \mathbf{0} & \mathbf{0} & \mathbf{0}
    \end{bmatrix}.
\end{equation}
Hence, we obtain the final result:
\begin{equation}
    \mathrm{TF}(\mathbf{P})  \approx \mathbf{U} + \mathbf{T}= \begin{bmatrix}
        \mathbf{0} & \mathbf{b} & \mathbf{a} \\
        \mathbf{0} & \mathbf{0} & \mathbf{0} \\
        \mathbf{e}_1 & \mathbf{e}_2 & \mathbf{e}_3
    \end{bmatrix}.
\end{equation}
% This shows that the FFN can be constructed to operate specific rows as needed. By applying $\mathbf{f}$ column-wise to \( \mathrm{Attn}(\mathbf{P}) \), we obtain the desired result:
% \begin{equation}
%     \text{TF}(\mathbf{P})=\begin{bmatrix}
%         \textbf{0}&\mathbf{b}&\mathbf{a}\\
%  \textbf{0}&\textbf{0}&\textbf{0}\\
% \mathbf{e}_1&\mathbf{e}_{2}&\mathbf{e}_{3}
%     \end{bmatrix}.
% \end{equation}
The rest of the statements in (ii) can be proved similarly. 

For (iii), the input prompt is depicted as follows:
    $$ \mathbf{P}=\begin{bmatrix}
         {0}&&\mathbf{a}^{\top} & \\
        {0}&&\mathbf{b}^{\top} & \\
        {0}& 0 & \dots & {0}\\
        \mathbf{e}_{1}& \mathbf{e}_2 & \dots& \mathbf{e}_{n+1}
    \end{bmatrix} .    
    $$
    We construct: 
$$\mathbf{W}_{Q}=\begin{bmatrix}
    \mathbf{0}&\mathbf{0}&c\\
    \mathbf{0}&\mathbf{1}&1
\end{bmatrix}, \quad \mathbf{W}_{Q}\mathbf{P}=\begin{bmatrix}
    0&\dots&0&c\\
    1 & \dots&1&1
\end{bmatrix},
\quad \mathbf{W}_{K}=\begin{bmatrix}
    1&\mathbf{0}&0&\mathbf{0}\\
    0&\mathbf{0}&C&\mathbf{0}
\end{bmatrix},
$$
$$\mathbf{W}_{K}\mathbf{P}=\begin{bmatrix}
    0&\mathbf{a}^{\top} \\
   C&\mathbf{0} \\
\end{bmatrix}, \quad \mathbf{W}_{V}=\begin{bmatrix}
    0&0&\mathbf{0}\\
    0&1&\mathbf{0}\\
    \mathbf{0}&\mathbf{0}&\mathbf{0}
\end{bmatrix}, \quad \mathbf{W}_{V}\mathbf{P}=e^{C}\begin{bmatrix}
\mathbf{0}&\mathbf{0}\\
0&\mathbf{b}^{\top}\\
\mathbf{0}&\mathbf{0}
\end{bmatrix},
$$
    for positive constants $C,\ c$. By using similar argument in (ii), 
    we have 
  the desired results:  
 \begin{equation}
      \text{TF}(\mathbf{P})\approx        \begin{bmatrix}
\mathbf{0}& &\mathbf{a}^{\top} & \\
        \mathbf{0}&&\mathbf{b}^{\top}& \\
        {0}& 0 & \dots & \mathbf{a}^{\top}\mathbf{b}\\
        \mathbf{e}_{1}& \mathbf{e}_2 & \dots&  \mathbf{e}_{n+1}
        \end{bmatrix}.
 \end{equation}
 The rest of the statements in (iii) can be proved similarly. 
 
 For (iv), the input prompt is depicted as follows:
    $$ \mathbf{P}=\begin{bmatrix}
        \mathbf{0}& &\mathbf{A} &  \\
        \mathbf{0}&  &\mathbf{0} &  \\
        \mathbf{e}_{1}& \mathbf{e}_{2}&\dots&\mathbf{e}_{n+1}
    \end{bmatrix}.    $$
We construct 
$$\mathbf{W}_{Q}=\begin{bmatrix}
    \mathbf{0}&\mathbf{0}&c\mathbf{E}_{n}\\
    \mathbf{0}&1&\mathbf{1}
\end{bmatrix},\mathbf{W}_{Q}\mathbf{P}=\begin{bmatrix}
    \mathbf{0}&c\mathbf{E}_{n}\\
    1 &\mathbf{1}
\end{bmatrix}, \quad \mathbf{W}_{K}=\begin{bmatrix}
    \mathbf{E}_{n}&\mathbf{0}&\mathbf{0}&\mathbf{0}\\
    \mathbf{0}&\mathbf{0}&C&\mathbf{0}
\end{bmatrix},
$$
$$
\mathbf{W}_{K}\mathbf{P}=\begin{bmatrix}
    \mathbf{0}& &\mathbf{A} &  \\
   C&0 &\dots &  0\\
\end{bmatrix}, \quad \mathbf{W}_{V}=e^{C}\begin{bmatrix}
\mathbf{0}&\mathbf{0}\\
    \mathbf{0}&\mathbf{E}_{n}\\
    \mathbf{0}&\mathbf{0}
\end{bmatrix}, \quad \mathbf{W}_{V}\mathbf{P}=e^{C}\begin{bmatrix}
\mathbf{0}&\mathbf{0}\\
    \mathbf{0}&\mathbf{E}_{n}\\
    \mathbf{0}&\mathbf{0}
\end{bmatrix},$$
for positive constants $C,\ c$. By using deduction similar to (ii) and the results in 
\cite{gao2024expressive,Giannou2023LoopedTA}, we have the desired result:
\begin{equation}
     \text{TF}(\mathbf{P})\approx\begin{bmatrix}
        \mathbf{0}&&\mathbf{A}^{\top}&\\
         \mathbf{0}&& \mathbf{0}&\\
        \mathbf{e}_{1}&\mathbf{e}_{2}&\dots&\mathbf{e}_{n+1}
    \end{bmatrix}.
\end{equation}
The rest of the statements in (iv) can be proved similarly.
\end{proof}

Next, we briefly describe the requirements of the loop-transformer, as summarized in \cref{Tablecopmpare}. First, according to Theorem 4 in the paper \cite{Giannou2023LoopedTA}, performing $+,-,\cdot,\div$ between two elements requires an input size of $O(n)$ and 12 layers. Thus, performing full pointwise operations ($+,-,\cdot,\div$) between two vectors of length $n$ requires an input size of $O(n)$ and $12n$ layers. For row shifting, according to Lemmas 2 and 3 in the paper \cite{Giannou2023LoopedTA}, swapping a single element between two row vectors requires two read and write operations, each with an input size of $O(n \log n)$ and one layer. Therefore, performing a full row swap requires an input size of $O(n \log n)$ and 4n layers of the transformer. The column shift can be deduced similarly. According to Lemma 6 in \cite{Giannou2023LoopedTA}, matrix transposition requires an input size of $O(n^3)$ and four layers of the transformer. For matrix multiplication $\mathbf{A}^{\top}\mathbf{B}$, according to Lemma 20 and Lemma 6 in \cite{Giannou2023LoopedTA}, the matrix $\mathbf{B}$ must be first transposed before matrix multiplication to obtain the result. Therefore, the final requirement is an input size of $O(n^3)$ and 6 layers. The requirement of $\mathbf{A}\mathbf{B}$ can be deduced similarly. For all other vector operations, the vector is first zero-padded into a square matrix, and then matrix theorems are applied; thus, the requirements are similar to those for matrices.

According to \cref{TheAlgformerTheorembasicopr}, we demonstrate that the transformer architecture is sufficiently expressive to represent essential numerical operators. This finding confirms that transformers can effectively embody fundamental numerical linear algebra operations within a single cohesive framework, highlighting their potential as a unified solver architecture.
NLAFormer discards the simulation of the computer control flow adopted by loop-transformer \cite{Giannou2023LoopedTA}, significantly reducing both the size of the input matrix and the required number of layers. 

\section{Learning the Conjugate Gradient Algorithm}
\label{sec:Evaluating NLAFormer}

This section demonstrates how NLAFormer can assemble numerical linear algebra operations and learn the conjugate gradient algorithm for solving symmetric positive definite systems.
Solving systems of linear equations is a fundamental task in numerical linear algebra. The objective is to find the solution vector $\mathbf{x}$ that satisfies $\mathbf{Ax} = \mathbf{b}$, where $\mathbf{A} = [\mathbf{a}_1, \dots, \mathbf{a}_n]^T \in \mathbb{R}^{n \times n}$ is a positive definite matrix, and $\mathbf{b} = [b_1, \dots, b_n]^T \in \mathbb{R}^n$ is the corresponding vector on the right side.
The iterative process of the CG method is as follows:
\begin{equation}
\label{CGalg}
\begin{aligned}
 \mathbf{d}_{0}&=\mathbf{r}_{0}=\mathbf{b}-\mathbf{A}\mathbf{x}_{0},\quad k=0  , \\
 \alpha_{k}&=\frac{\left\| \mathbf{r}_{k}\right \|^2}{\mathbf{d}_{k}^{T}A\mathbf{d}_{k}} ,\quad \mathbf{x}_{k+1}=\mathbf{x}_{k}+\alpha_{k}\mathbf{d}_{k},\\
 \mathbf{r}_{k+1}&=\mathbf{r}_{k}-\alpha_{k}\mathbf{A}\mathbf{d}_{k},\quad \beta_{k+1}=\frac{\left\| \mathbf{r}_{k+1}\right \|^2}{\left\| \mathbf{r}_{k}\right \|^2},\\
 \mathbf{d}_{k+1}&=\mathbf{r}_{k+1}+\beta_{k+1}\mathbf{d}_{k},\quad k=k+1,
\end{aligned}
\end{equation}
as introduced in \cite{shewchuk1994introduction}.

The conjugate gradient method inherently follows an iterative dependency structure, with each step building on the results of previous iterations.
To align with the workflow of iterative numerical algorithms, NLAFormer is composed of three sequential transformer blocks: a pre-processing block ($\text{TF}_{\text{pre}}$), an iterative loop block ($\text{TF}_{\text{loop}}$), and a post-processing block ($\text{TF}_{\text{post}}$). The complete architecture is deonated as $\text{NLAF}$ and is defined as follows:
\begin{equation}
\label{TFstructure}
\text{NLAF}(\mathbf{P})= \text{TF}_{\text{post}} \Big( \underbrace{\text{TF}_{\text{loop}} \big( \cdots \text{TF}_{\text{loop}}}_{\text{ looping}} ( \text{TF}_{\text{pre}} (\mathbf{P}) ) \cdots \big) \Big).
\end{equation}
The truncated variants are denoted by $\text{NLAF}^t$, for $t = 0, \dots, T$. 
Specifically, $\text{NLAF}^0$ includes only the pre-processing block $\text{TF}_{\text{pre}}$. For $1 \le t \le T - 1$, $\text{NLAF}^t$ consists of $\text{TF}_{\text{pre}}$ followed by $t$ applications of the loop block $\text{TF}_{\text{loop}}$. Finally, $\text{NLAF}^T$ executes the complete pipeline.
An illustration of the structured iterative model is provided in \cref{Structured iterative model diagram}.

\begin{figure}[h]
    \centering
    \includegraphics[width=1\linewidth]{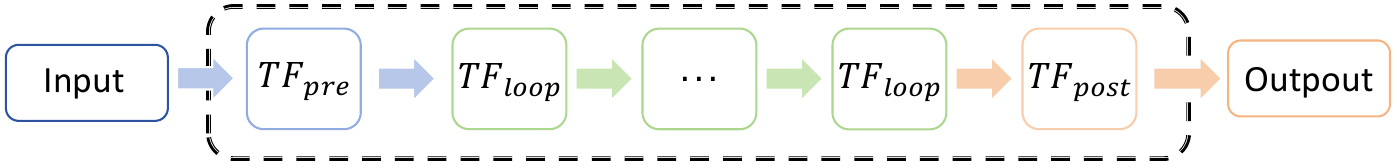}
    \caption{Structured iterative model diagram}
    \label{Structured iterative model diagram}
\end{figure}

The preprocessing module corresponds to the initialization and input data preparation steps in the algorithm, serving to prepare the data for subsequent processing. The loop module performs multiple iterations, shares consistent parameters between them, and encapsulates the core iterative computations. Finally, the postprocessing module performs the final iteration and outputs the result in the required format. The effectiveness of this structured iterative module design has been validated in the paper \cite{gao2024expressive}.
This design mimics the control flow of traditional iterative solvers but differs from conventional approaches that rely on explicitly hand-coded rules. Instead, each transformer module learns to represent the underlying update rules by observing numerical patterns during training. This ability to learn underlying procedures from data is precisely the capability that we seek to develop for tasks in numerical linear algebra. 

The following theorem demonstrates that there exists an NLAFormer capable of expressing the conjugate gradient algorithm. 

\begin{theorem}
\label{TheAlgformer4CG}
There exists a specially crafted transformer module, comprised of $TF_{pre}$ (a single-layer, four-head transformer) and $TF_{loop}$ (a single-layer, two-head transformer), that can execute the conjugate gradient method, as outlined in Equation \cref{CGalg}, for solving linear systems.
\end{theorem}

%The proof of Theorem 3.1 will be given later. 
Let us first establish the following results that a single-layer, two-head transformer is capable of executing a single step of the conjugate gradient method when employed with a specific configuration of input. 

\begin{lemma}
\label{lemma1StepCG}
A single-layer, dual-head transformer has the capability to execute a single iteration of
the conjugate gradient method for solving linear systems as described in Equation \cref{CGalg}.
\end{lemma}

\begin{proof}
    We consider the following input prompt:
    \begin{equation}
            \mathbf{P}=\begin{bmatrix}
        \textbf{0} &\mathbf{a}_{1} &\mathbf{a}_{2} &\dots & \mathbf{a}_{n-1}& \mathbf{a}_{n}\\
        0 &b_{1} &b_{2} &\dots & b_{n-1}& b_{n}\\
        \textbf{0} &\textbf{0} &\textbf{0} & \dots &\textbf{0} &\mathbf{d}_{k} \\
        \textbf{0} & \textbf{0} &\textbf{0} &\dots &\textbf{0}&\mathbf{x}_{k}\\
        \textbf{0} & \textbf{0} &\textbf{0} &\dots &\textbf{0}&\mathbf{r}_{k}\\
                 \textbf{0}& \textbf{0} &\textbf{0}&\dots & \textbf{0}&\textbf{0}\\
         \mathbf{e}_{1} &\mathbf{e}_{2} &\mathbf{e}_{3}& \dots&\mathbf{e}_{n}&\mathbf{e}_{n+1}\\
    \end{bmatrix}. 
    \end{equation}
We construct 
$$
\mathbf{W}_{Q}=
\begin{bmatrix}
    \mathbf{0}&c\mathbf{E}_{n}&\mathbf{0}&\mathbf{0}\\
    \mathbf{0}&\mathbf{0}&\mathbf{0}&\mathbf{1}
\end{bmatrix}, \quad \mathbf{W}_{Q}\mathbf{P}=\begin{bmatrix}
    \mathbf{0}&c\mathbf{d}_{k}\\
    \mathbf{1} & 1
\end{bmatrix}, \quad \mathbf{W}_{K}=
\begin{bmatrix}
    \mathbf{E}_{n}&\mathbf{0}&\mathbf{0}&\mathbf{0}\\
    \mathbf{0}&\mathbf{0}&C&\mathbf{0}
\end{bmatrix},
$$
$$
\mathbf{W}_{K}\mathbf{P}=\begin{bmatrix}
    0&\mathbf{a}_{1} &\dots &\mathbf{a}_{n}\\
    C&0&0&0
\end{bmatrix}, \quad 
\mathbf{W}_{V}=\begin{bmatrix}
    \mathbf{0}&\mathbf{0}\\
    \mathbf{0}&e^{C}\mathbf{E}_{n}\\
    \mathbf{0}&\mathbf{0}
\end{bmatrix}, \quad 
\mathbf{W}_{V}\mathbf{P}=e^{C}\begin{bmatrix}
    \mathbf{0}&\mathbf{0}\\
    \mathbf{0}&e^{C}\mathbf{E}_{n}\\
    \mathbf{0}&\mathbf{0}
\end{bmatrix},
$$ 
for positive constants $C,\ c$. Also, we denote
$$
\mathbf{Z}=\mathbf{P}^{T}\mathbf{W}^{T}_{K}\mathbf{W}_{Q}\mathbf{P}=\begin{bmatrix}
C & C & \dots & C & C\\
    0&0&\dots&0 &c\mathbf{a}_{1}^{T}\mathbf{d}_{k}\\
    \vdots&\vdots&\ddots&\vdots&\vdots\\
    0&0&\dots &0&c\mathbf{a}_{n}^{T}\mathbf{d}_{k}
\end{bmatrix}.
$$
By using similar deduction in \cite{gao2024expressive,Giannou2023LoopedTA}, we know the following equation holds for $2\leq l\leq n+1$:
\begin{equation}
    (e^{C}\text{softmax}(\mathbf{Z}))_{l,n+1}\approx1+{Z}_{l,n+1},
\end{equation}
as long as $C$ is sufficiently large and $c$ is sufficiently small. After we introduce another head to cancel the constant \cite{gao2024expressive}, we know the output of the attention layer can be given as follows:
\begin{equation}
    \sum_{h=1}^{2}\mathbf{W}_{V}^{(h)}\mathbf{P} \cdot \text{softmax}(\mathbf{P}^{T}\mathbf{W}^{(h)T}_{K}\mathbf{W}_{Q}^{(h)}\mathbf{P})\approx c\begin{bmatrix}
    \mathbf{0}&\mathbf{0}\\
    \mathbf{0}&\mathbf{A}\mathbf{d}_{k}\\
    \mathbf{0}&\mathbf{0}
\end{bmatrix}.
\end{equation}
We denote $\text{Attn}(\mathbf{P})\approx\mathbf{U}$, where
\begin{equation}
    \mathbf{U}=\begin{bmatrix}
        \textbf{0}&\mathbf{a}_{1} &\mathbf{a}_{2} &\dots & \mathbf{a}_{n-1}& \mathbf{a}_{n}\\
        0&b_{1} &b_{2} &\dots & b_{n-1}& b_{n}\\
        \textbf{0}&\textbf{0} &\textbf{0} & \dots &\textbf{0} &\mathbf{d}_{k} \\
        \textbf{0}&\textbf{0} & \textbf{0} &\dots &\textbf{0}&\mathbf{x}_{k}\\
        \textbf{0}&\textbf{0} & \textbf{0} &\dots &\textbf{0}&\mathbf{r}_{k}\\
        \textbf{0}&\textbf{0}&\textbf{0}&\dots&\textbf{0}&c\mathbf{A}\mathbf{d}_{k}\\
         \mathbf{e}_{1} &\mathbf{e}_{2} &\mathbf{e}_{3}& \dots&\mathbf{e}_{n}&\mathbf{e}_{n+1}
    \end{bmatrix}.
\end{equation}
It is evident that all of $\mathbf{x}_{k+1} - \mathbf{x}_{k}$, $\mathbf{r}_{k+1} - \mathbf{r}_{k}$, and $\mathbf{d}_{k+1} - \mathbf{d}_{k}$ are results of continuous mappings from $\mathbf{x}_{k}$, $\mathbf{r}_{k}$, $\mathbf{d}_{k}$, and $\mathbf{A}\mathbf{d}_{k}$, respectively. Therefore, similar to the proof of Theorem 1, for all $\varepsilon > 0$, the following approximation holds according to the universal approximation theorem \cite[Theorem 2]{leshno1993multilayer}:
\begin{equation}
\left\| \text{FFN}(\textbf{U}) - \textbf{T} \right\| < \varepsilon,\     \mathbf{T}=\begin{bmatrix}
        \textbf{0}&\textbf{0} &\textbf{0} &\dots & \textbf{0}& \textbf{0}\\
        0&0 &0 &\dots & 0& 0\\
        \textbf{0}&\textbf{0} &\textbf{0} & \dots &\textbf{0} &\mathbf{d}_{k+1} - \mathbf{d}_{k} \\
        \textbf{0}&\textbf{0} & \textbf{0} &\dots &\textbf{0}&\mathbf{x}_{k+1} - \mathbf{x}_{k}\\
        \textbf{0}&\textbf{0} & \textbf{0} &\dots &\textbf{0}&\mathbf{r}_{k+1} - \mathbf{r}_{k}\\
        \textbf{0}&\textbf{0}&\textbf{0}&\dots&\textbf{0}&-c\mathbf{A}\mathbf{d}_{k}\\
         \textbf{0} &\textbf{0} &\textbf{0}& \dots&\textbf{0}&\textbf{0}
    \end{bmatrix}.
\end{equation}
Therefore, the output of the transformer layer can be represented as follows:
    \begin{equation}
            \text{TF}_{loop}(\mathbf{P}) \approx \mathbf{U} + \mathbf{T}  =\begin{bmatrix}
        \textbf{0} &\mathbf{a}_{1} &\mathbf{a}_{2} &\dots & \mathbf{a}_{n-1}& \mathbf{a}_{n}\\
        0 &b_{1} &b_{2} &\dots & b_{n-1}& b_{n}\\
        \textbf{0} &\textbf{0} &\textbf{0} & \dots &\textbf{0} &\mathbf{d}_{k+1} \\
        \textbf{0} & \textbf{0} &\textbf{0} &\dots &\textbf{0}&\mathbf{x}_{k+1}\\
        \textbf{0} & \textbf{0} &\textbf{0} &\dots &\textbf{0}&\mathbf{r}_{k+1}\\
                 \textbf{0}& \textbf{0} &\textbf{0}&\dots & \textbf{0}&\textbf{0}\\
        \mathbf{e}_{1} &\mathbf{e}_{2} &\mathbf{e}_{3}& \dots&\mathbf{e}_{n}&\mathbf{e}_{n+1}\\
    \end{bmatrix}. 
    \end{equation}
    This concludes that the specifically crafted transformer is capable of implementing a single iteration of the conjugate gradient method as stated in \cref{CGalg}.
\end{proof}

We now proceed to prove Theorem 3.1.

\begin{proof}[Proof for \cref{TheAlgformer4CG}]
     The initial input is given as follows:
    \begin{equation}
            \mathbf{P}=\begin{bmatrix}
        \textbf{0}&\mathbf{a}_{1} &\mathbf{a}_{2} &\dots & \mathbf{a}_{n-1}& \mathbf{a}_{n}\\
        0&b_{1} &b_{2} &\dots & b_{n-1}& b_{n}\\
        \textbf{0} &\textbf{0} & \textbf{0}&\dots &\textbf{0} &\textbf{0} \\
        \textbf{0} & \textbf{0} &\textbf{0}&\dots &\textbf{0}&\mathbf{x}_{0}\\
        \textbf{0} & \textbf{0} &\textbf{0}&\dots &\textbf{0}&\textbf{0}\\
        \mathbf{e}_{1} &\mathbf{e}_{2} &\mathbf{e}_{3}& \dots&\mathbf{e}_{n}&\mathbf{e}_{n+1}\\
    \end{bmatrix}. 
    \end{equation}
 We construct 
$$\mathbf{W}_{Q}=\begin{bmatrix}
    \mathbf{0}&\mathbf{0}&&{1}\\
    \mathbf{0}&\mathbf{1}&&{1}
\end{bmatrix}, \quad \mathbf{W}_{Q}\mathbf{P}=\begin{bmatrix}
    0 &0 & \dots &0&1\\
    1 & 1&1&1&1
\end{bmatrix}, \quad 
\mathbf{W}_{K}=\begin{bmatrix}
    \mathbf{0}&c&\mathbf{0}&{0}&\mathbf{0}\\
    \mathbf{0}&0&\mathbf{0}&{C}&\mathbf{0}
\end{bmatrix},
$$
$$
\mathbf{W}_{K}\mathbf{P}=\begin{bmatrix}
    0&cb_{1} &cb_{2} &\dots & cb_{n}\\
    C&0&0&0&0
\end{bmatrix}, \quad
\mathbf{W}_{V}=\begin{bmatrix}
    \mathbf{0}&\mathbf{0}\\
    \mathbf{0}&e^{C}\mathbf{E}_{n}\\
    \mathbf{0}&\mathbf{0}
\end{bmatrix}, \quad 
\mathbf{W}_{V}\mathbf{P}=e^{C}\begin{bmatrix}
    \mathbf{0}&\mathbf{0}\\
    \mathbf{0}&\mathbf{E}_{n}\\
    \mathbf{0}&\mathbf{0}
\end{bmatrix},
$$ 
for positive constants $C,\ c$. Also, denote 
$$
\mathbf{Z}=\mathbf{P}^{T}\mathbf{W}^{T}_{K}\mathbf{W}_{Q}\mathbf{P}=\begin{bmatrix}
    C & C & \dots & C & C\\
    0&0&\dots &0&cb_{1}\\
    \vdots&\vdots&\ddots&\vdots&\vdots\\
    0&0&\dots &0&cb_{n}\\
\end{bmatrix}.
$$
We know the following equation holds for $2\leq l\leq n+1$:
\begin{equation}
    (e^{C}\text{softmax}(\mathbf{Z}))_{l,n+1}\approx1+{Z}_{l,n+1},
\end{equation}
as long as $C$ is sufficiently large and $c$ is sufficiently small. After we introduce another head to cancel the constant \cite{gao2024expressive}, we know the output of the attention layer can be:
\begin{equation}
\label{HeadMakeB}
    \sum_{h=1}^{2}\mathbf{W}_{V}^{(h)}\mathbf{P} \cdot \text{softmax}(\mathbf{P}^{T}\mathbf{W}^{(h),T}_{K}\mathbf{W}_{Q}^{(h)}\mathbf{P})\approx c\begin{bmatrix}
    \mathbf{0}&\mathbf{0}\\
    \mathbf{0}&\mathbf{b}\\
    \mathbf{0}&\mathbf{0}
\end{bmatrix}.
\end{equation}
Subsequently, following the procedure outlined in \cref{lemma1StepCG}, we employ two additional attention heads to obtain the following result:
\begin{equation}
\label{HeadMakeAx}
    \sum_{h=3}^{4}\mathbf{W}_{V}^{(h)}\mathbf{P} \cdot \text{softmax}(\mathbf{P}^{T}\mathbf{W}^{(h),T}_{K}\mathbf{W}_{Q}^{(h)}\mathbf{P})\approx c\begin{bmatrix}
    \mathbf{0}&\mathbf{0}\\
    \mathbf{0}&\mathbf{A}\mathbf{x}_{0}\\
    \mathbf{0}&\mathbf{0}
\end{bmatrix}.
\end{equation}
Referring to the identity: $\mathbf{d}_{0}=\mathbf{r}_{0}=\mathbf{b}-\mathbf{A}\mathbf{x}_{0}$, we confirm that the required quantities have been obtained through attention heads, thus yielding the following results after passing through the FFN:
    \begin{equation}
            \begin{bmatrix}
       \textbf{0}_{n,1}& \mathbf{a}_{1} &\mathbf{a}_{2} &\dots & \mathbf{a}_{n-1}& \mathbf{a}_{n}\\
        0&b_{1} &b_{2} &\dots & b_{n-1}& b_{n}\\
        \textbf{0}&\textbf{0} &\textbf{0} & \dots &\textbf{0} &\mathbf{d}_{0} \\
        \textbf{0}&\textbf{0} & \textbf{0} &\dots &\textbf{0}&\mathbf{x}_{0}\\
        \textbf{0}&\textbf{0} & \textbf{0} &\dots &\textbf{0}&\mathbf{r}_{0}\\
         \textbf{0}& \textbf{0} &\textbf{0}&\dots & \textbf{0}&\textbf{0}\\
        \mathbf{e}_{1} &\mathbf{e}_{2} &\mathbf{e}_{3}& \dots&\mathbf{e}_{n}&\mathbf{e}_{n+1}\\
    \end{bmatrix}, 
    \end{equation}
    By applying \cref{lemma1StepCG} iteratively, after $n$ applications of the looped transformer, the resulting matrix adheres to the following pattern:
        \begin{equation}
           \begin{bmatrix}
        \textbf{0}&\mathbf{a}_{1} &\mathbf{a}_{2} &\dots & \mathbf{a}_{n-1}& \mathbf{a}_{n}\\
       0& b_{1} &b_{2} &\dots & b_{n-1}& b_{n}\\
        \textbf{0}&\textbf{0} &\textbf{0} & \dots &\textbf{0} &\mathbf{d}_{n} \\
        \textbf{0}&\textbf{0} & \textbf{0} &\dots &\textbf{0}&\mathbf{x}_{n}\\
        \textbf{0}&\textbf{0} & \textbf{0} &\dots &\textbf{0}&\mathbf{r}_{n}\\
                 \textbf{0}& \textbf{0} &\textbf{0}&\dots & \textbf{0}&\textbf{0}\\
        \mathbf{e}_{1} &\mathbf{e}_{2} &\mathbf{e}_{3}& \dots&\mathbf{e}_{n}&\mathbf{e}_{n+1}\\
    \end{bmatrix}, 
    \end{equation}
    which completes the proof.
\end{proof}

We note that NLAFormer effectively captures iterative procedures, confirming its utility not only for isolated algebraic operations but also for structured iterative algorithms. 
Note that a sequence of basic operators required by the algorithm is effectively assembled to complete the construction. Similarly, for other iterative methods, the relevant basic operators can be composed to support their representation. Therefore, our detailed demonstration of the conjugate gradient method serves as a representative example that clearly illustrates how the logic of iterative algorithms can be systematically expressed within NLAFormer. This methodological clarity enables related algorithms to be constructed in a step-by-step manner. It facilitates the structured and intuitive integration of future algorithmic extensions and reinforces NLAFormer as a general and expressive computational paradigm.

\section{Experimental Results}
\label{ExperimentP1}

In this section, we present 
empirical evaluations of the NLAFormer.
In the first part, we evaluate whether NLAFormer can effectively learn iterative solutions of the CG method.
In the second part, we investigate whether NLAFormer can autonomously discover iterative procedures that 
achieve faster convergenc than the CG method.

\begin{figure}[htbp]
    \centering
    \includegraphics[width=1\linewidth]{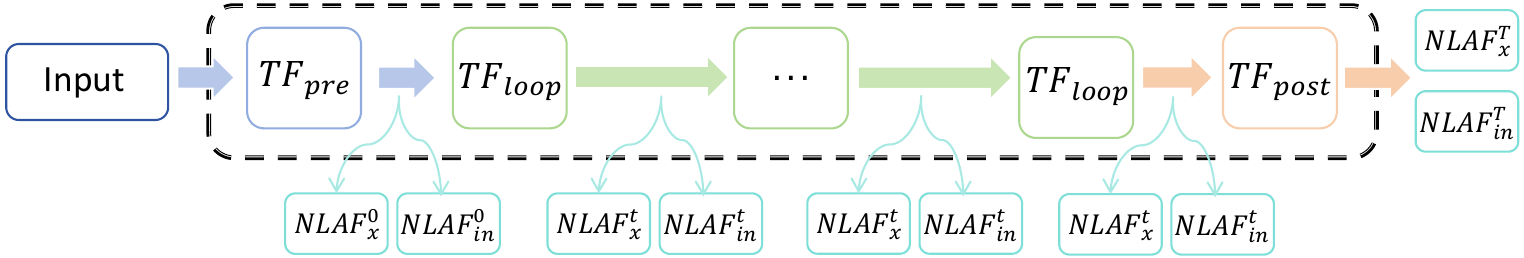}
    \caption{Diagram of the probe mechanism.}
    \label{Diagram of the probe mechanism}
\end{figure}

\subsection{The First Setting}

We introduce two probes to extract variable information from NLAFormer's iterative output at each step, as depicted in \cref{Diagram of the probe mechanism}. 
The first probe captures the iterative solution ($\text{NLAF}_x^t \in \mathbb{R}^n$), corresponding to CG’s approximation of the solution at each step, while the second probe extracts intermediate variables ($\text{NLAF}_{\text{in}}^t \in \mathbb{R}^{2n}$), corresponding to CG’s residual ($\mathbf{r}_t$) and direction ($\mathbf{d}_t$) vectors.
Moreover, we implement two training approaches to investigate the extent to which the model internalizes CG's iterative logic.
The first one ({\bf Result supervision}) is that 
only iterative solutions ($\mathrm{NLAF}_x^t$) are supervised with CG results whose loss is given as follows:
$$
   \min_{\mathbf{\Theta}}  \mathbb{E}_ {\mathbf{P}}  \left [ 
  \frac {1}{n(T+1)}  \sum_ {t=0}^ {T}  \left\|\text{NLAF}_ {\text{x}}^ {t}  (\mathbf{P};\Theta)-\mathbf{x}_{t}\right\|^{2}_{2}
  \right ].
  $$
  The second one ({\bf Joint supervision}) is that 
 both iterative solutions ($\mathrm{NLAF}_x^t$) and 
 intermediate variables ($\mathrm{NLAF}_{in}^t$) are explicitly supervised using the full loss function defined as follows:
\begin{equation}
    \label{train_loss_result}
      \min_{\mathbf{\Theta}}  \mathbb{E}_ {\mathbf{P}}  \left [  
  \frac {1}{n(T+1)}  \left (\sum _ {t=0}^ {T}  \|\text{NLAF}_ {\text{x}}^ {t}  (\mathbf{P};\Theta)-\mathbf{x}_{t}\|^{2}_{2}+\eta\|\text{NLAF}_ {\text{in}}^ {t}  (\mathbf{P};\Theta)-(\mathbf{r}_{t},\mathbf{d}_{t})\|^{2}_{2} \right ) \right ], 
\end{equation}
where $(\mathbf{r}_{t},\mathbf{d}_{t})$ are the intermediate variable of CG at the $t$th iterative in \cref{CGalg}, $\mathbf{P}$ is the input, $\Theta$ is the transformer parameter, and $\eta$ is a regularization parameter.

In our experiment, we utilize a linear system with dimension $n=20$, which is the same as the setting in the paper \cite{gao2024expressive}. This moderate scale facilitates a precise evaluation of the iterative learning capability of the model. 
The embedding size is $d=256$. The matrix $\mathbf{A}$ is constructed by adding a symmetric positive definite matrix, generated using standard scikit-learn libraries, to a diagonal matrix whose entries are independently drawn from a log normal distribution with zero mean and variance $\sigma = 1.2$. This construction ensures positive definiteness while easing the ill-conditioning, making the test instances more representative of realistic numerical settings. 
The vector $\mathbf{x}$ is sampled from the standard normal distribution, i.e., $\mathbf{x} \sim \mathcal{N}(\mathbf{0}, \mathbf{I})$. The model training employs an Adam optimizer \cite{kingma2014adam} with a dynamically adjusted learning rate schedule designed to achieve stable convergence over 700k epochs, as shown in \cref{fig:lrtrend}. 
The batch size of 64 was employed in the training procedure.

In \cref{fig:train loss}, we present a comparative analysis of the intermediate discrepancy between the beginning of training and after 100k training steps. The discrepancy calculates the difference between intermediate variables generated by the model ($\text{NLAF}_ {\text{in}}^t$) and those of the CG method ($\mathbf{r}_t,\mathbf{d}_t$), defined as follows: $\frac{1}{n}\sum_{t=2}^{n}\frac{1}{n-1}\|\text{NLAF}_{in}^{t}-(\mathbf{r}_{t},\mathbf{d}_{t})\|_{2}^{2}$.
We observe that, over training, intermediate variables spontaneously become increasingly aligned with the CG under result supervision, despite no explicit intermediate guidance. This behavior emerges without any external pressure to match the internal structure of CG, indicating that the NLAFormer is not arbitrarily generating intermediates, but rather finds a computational path that naturally aligns with the CG recursive structure. This convergence suggests that NLAFormer implicitly identifies a CG-like recursive structure from the provided data.

Next we conducted tests with the hyperparameter $\eta$ set to values in the range \{0.05, 0.01, 0.005, 0.001, 0.0005, 0.0001\} as shown in \cref{fig:diffeta}. It shows that effective training requires result supervision to dominate, but intermediate supervision should not be too weak. We ultimately selected 0.0005, which yielded the best performance, for use in the following 
experimental results. 
\begin{table}[htb]
\centering
\begin{tabular}{|c|c|c|c|}
\hline
                                                    Method                   & Initial epoch & 100k epoch & Percentage decrease \\ \hline
Joint supervision  & 46.8370      & 0.1920           & 99.7\%              \\
Result supervision & 45.9350      & 0.5058          & 98.0\%                \\ \hline
\end{tabular}
\caption{The discrepancy between the model-generated and the CG intermediate variables before and after training}
\label{fig:train loss}
\end{table}
\begin{figure}[H]
	\centering
 \subfigure[The trend of learning rate.]{
	\begin{minipage}{0.45\linewidth}
		\centering
		\includegraphics[width=\linewidth]{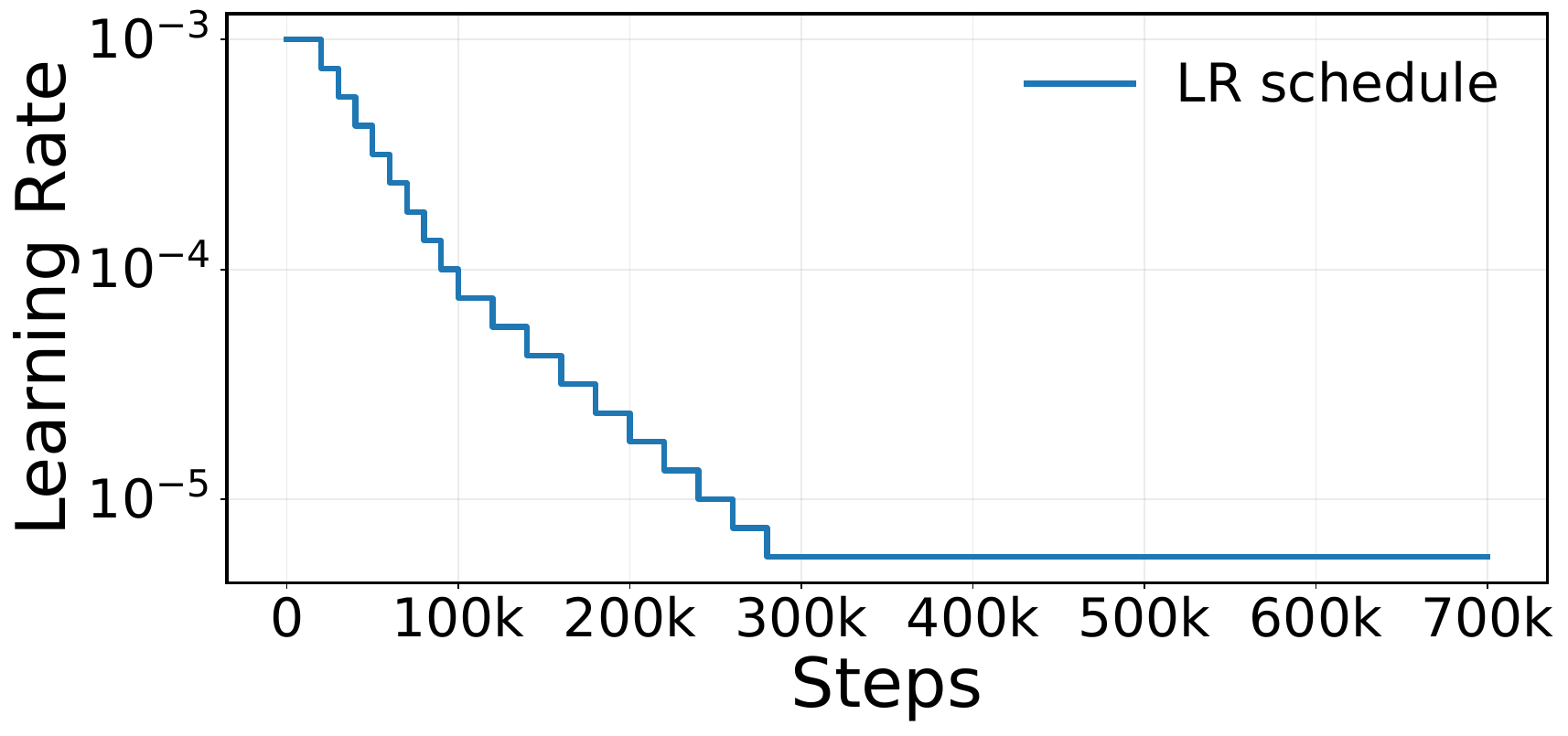}
%\caption{cloth}
%		\label{cloth}
\label{fig:lrtrend}
	\end{minipage}}
 \subfigure[The relative error with different $\eta$.]{
	\begin{minipage}{0.45\linewidth}
		\centering
		\includegraphics[width=\linewidth]{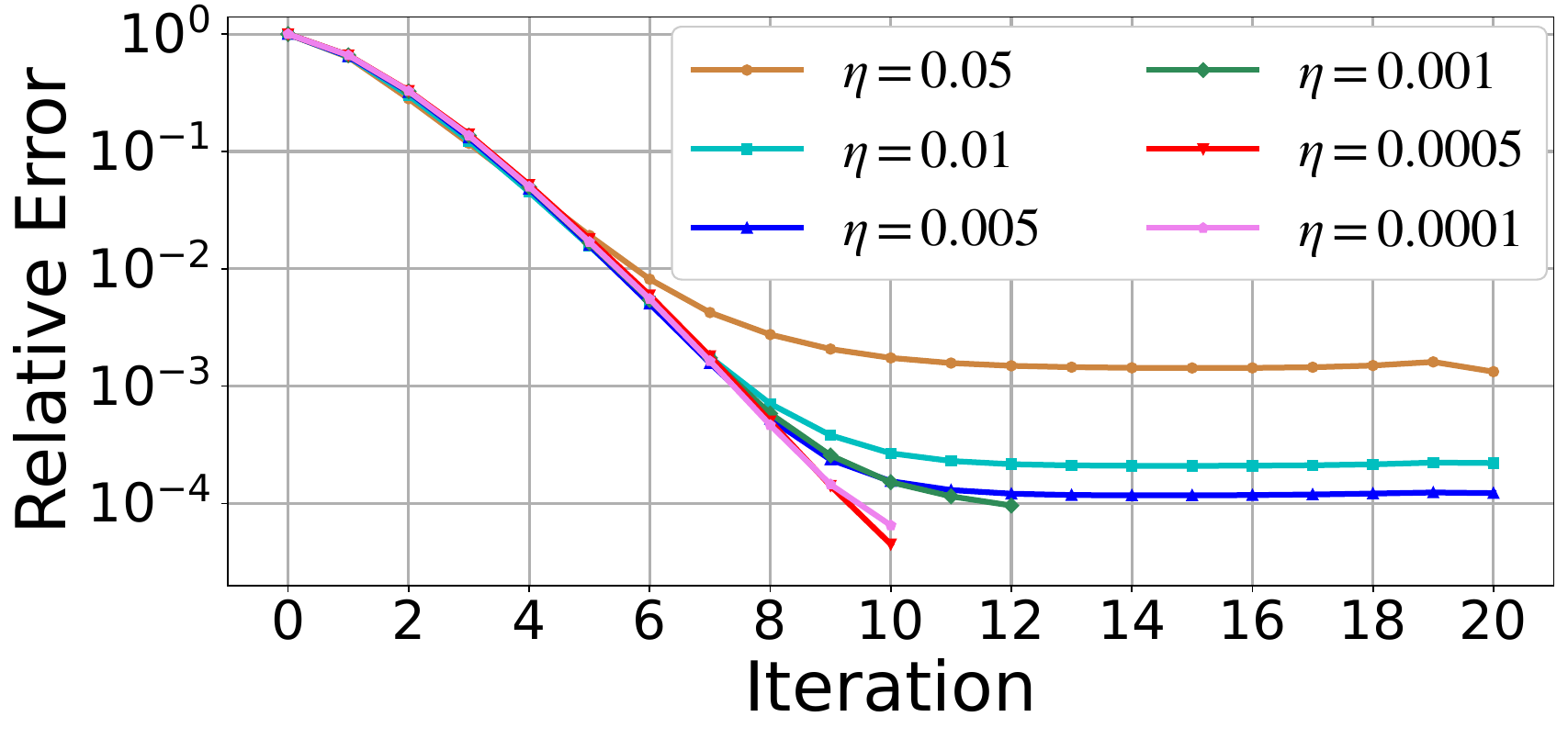}
        %diffeta.png
%\caption{cloth}
%		\label{cloth}
    \label{fig:diffeta}
	\end{minipage}}
   \caption{The experimental setting.}
  \label{}
\end{figure}
% \begin{figure}
%     \centering
%     \includegraphics[width=1\linewidth]{vecimg/vector_relative_error.pdf}
%     \caption{Caption}
%     \label{fig:placeholder}
% \end{figure}

\subsubsection{Experimental Results} 

After NLAFormer is trained, we demonstrate the testing performance of 
NLAFormer in \cref{fig:RE4exp1}.
In the figure, we display the average relative errors 
($\frac{\| \text{NLAF}_ {\text{x}}^ {t} -x_{true}\|^2}{\| x_{true}\|^2}$) 
over 512 tested linear systems against iterations. Here 
$x_{true}$ refers to the solution of a linear system.
According to the results, we see that 
NLAFormer can effectively mimic the iterative process of CG
across both types of supervision. 
Moreover, a quantitative comparison of intermediate variable errors, defined as 
$\frac{1}{n}\|\text{NLAF}_{in}^{t}-(\mathbf{r}_{t},\mathbf{d}_{t})\|_{2}^{2}$, is shown in \cref{fig:intermediate error}. The results demonstrate that joint supervision significantly enhances alignment with the true intermediate values of CG. The marked improvement under explicit intermediate guidance conclusively indicates that NLAFormer is not merely fitting input-output mappings but effectively internalizing the iterative logic of the CG method.
\begin{figure}[H]
	\centering
 \subfigure[The relative error curves of the iterative results with different training strategies.]{
	\begin{minipage}{0.45\linewidth}
		\centering
		\includegraphics[width=\linewidth]{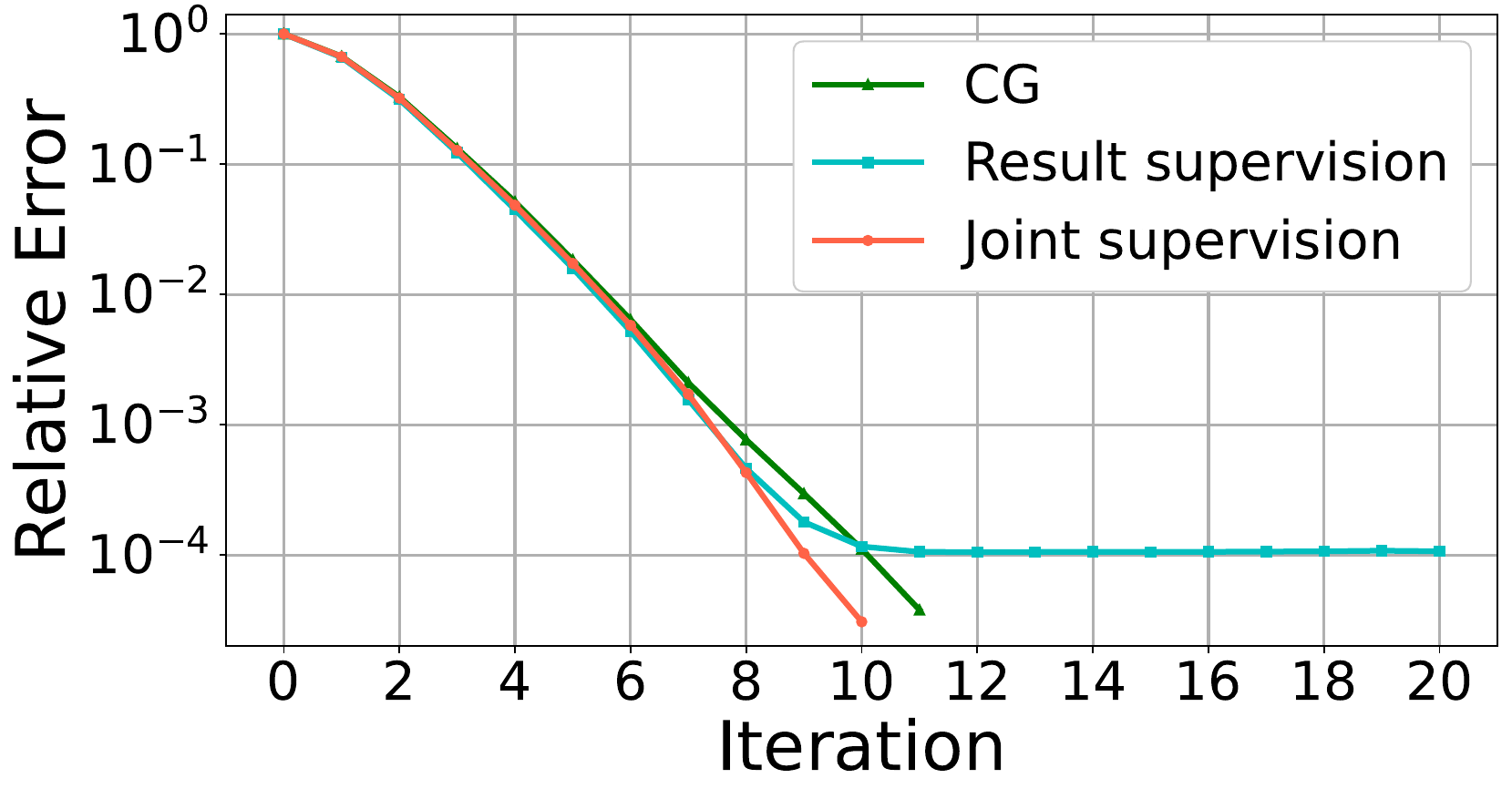}
        %img/RE4exp1.png
%\caption{cloth}
%		\label{cloth}
\label{fig:RE4exp1}
	\end{minipage}}
 \subfigure[The curves of intermediate variable errors with different training strategies.]{
	\begin{minipage}{0.45\linewidth}
		\centering
		\includegraphics[width=\linewidth]{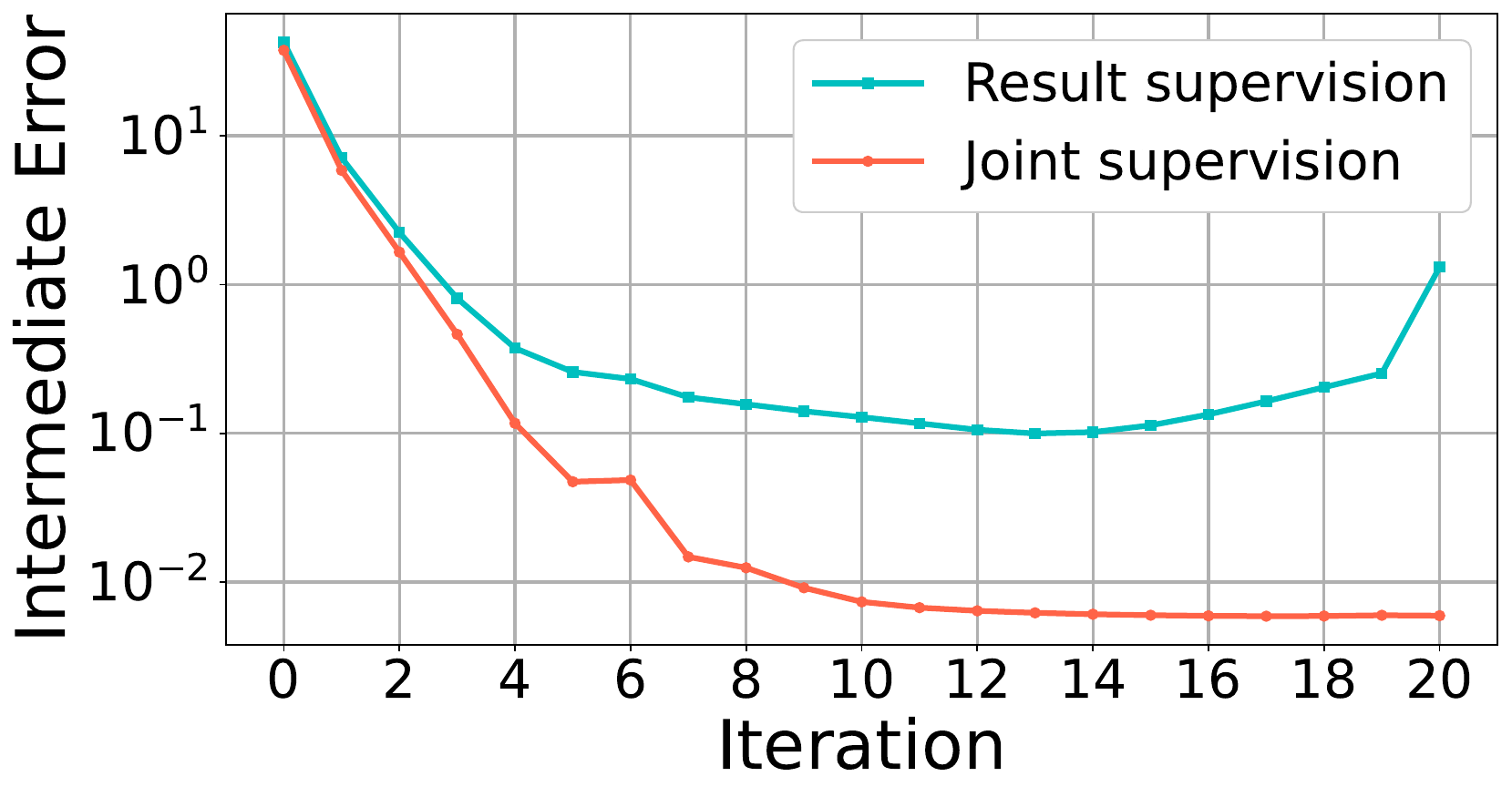}
%\caption{cloth}
%		\label{cloth}
    \label{fig:intermediate error}
	\end{minipage}}
   \caption{The experimental results.}
%  \label{}
\end{figure}
% \begin{figure}[htbp]
%     \centering
%     \includegraphics[width=0.7\textwidth]{img/RE4exp1.png} % 图片路径和大小
%     \caption{The relative error curves of the iterative results with different training strategies}
%     \label{fig:RE4exp1}
% \end{figure}

% \begin{figure}[htbp]
%     \centering
%     \includegraphics[width=0.7\textwidth]{img/intermediate error.png} % 图片路径和大小
%     \caption{The error between the generated intermediate variables and the CG intermediate variables}
%     \label{fig:intermediate error}
% \end{figure}

\subsection{The Second Setting}

Building upon our theoretical analysis, if a transformer is equipped with the ability to incorporate numerical algebraic operators as learnable components, it should be capable not only of replicating operators but also of learning to combine them in more effective ways to solve the equation. This suggests that, beyond simply imitating existing algorithms such as CG, NLAFormer has the potential to autonomously learn iterative procedures that lead to faster convergence compared to established solvers.
To support this exploration, we introduce losses composed of two parts. The first part, \textbf{Step supervision}, which guides the model to follow reference CG iteration steps, helping it internalize the logic of the CG:
$$
   \min_{\mathbf{\Theta}}  \mathbb{E}_ {\mathbf{P}}   
  \left [ \frac {1}{(T-T_ {0}+1)n}  \left (\sum _ {t=T_{0}}^ {T}  \|\text{NLAF}_ {\text{x}}^ {t}  (\mathbf{P};\Theta)-\mathbf{x}_{{t}}\|^{2}_{2} \right ) \right ],
  $$
  and the second part, 
    \textbf{Solution supervision}, which focuses on minimizing the distance to the solution of a linear system, encouraging the model to improve iteration efficiency. Their combination forms the final loss:
\begin{equation}
    \label{train_loss}
      \min_{\mathbf{\Theta}}  \mathbb{E}_ {\mathbf{P}} \left [   
  \frac {1}{(T-T_ {0}+1)n}  \left (\sum _ {t=T_{0}}^ {T}  \|\text{NLAF}_ {\text{x}}^ {t}  (\mathbf{P};\Theta)-\mathbf{x}_{{t}}\|^{2}_{2}+\lambda\|\text{NLAF}_ {\text{x}}^ {t}  (\mathbf{P};\Theta)-\mathbf{x}\|^{2}_{2}) \right ) \right ], 
\end{equation}
where $T_{0}=\max\{T-K+1,0\}$, with $K$ denoting the length of the learning window. It determines how many recent iterations are used during training and helps balance between learning efficiency and the model’s ability to capture long-term dynamics. $\text{NLAF}^{t}_{x}$ is the $t$-th iteration solution of the NLAFormer, as defined in \cref{TFstructure} and depicted in \cref{Diagram of the probe mechanism}, where $t=0,\dots, T$. The vector $\mathbf{x}$ is the underlying solution to the given 
linear system, and $\lambda$ is a regularization parameter. 

We conducted tests with the hyperparameter $\lambda$ set to values in the range \{13, 10, 7, 4, 1, 0.5\} as shown in \cref{fig:cgdifflmd}. It shows that increasing the weight of solution supervision can accelerate convergence, but an excessively large weight may degrade performance. We ultimately selected 10, which yielded the best performance. This value was used in the following experimental results.
\begin{figure}[htbp]
    \centering
    \includegraphics[width=0.5\textwidth]{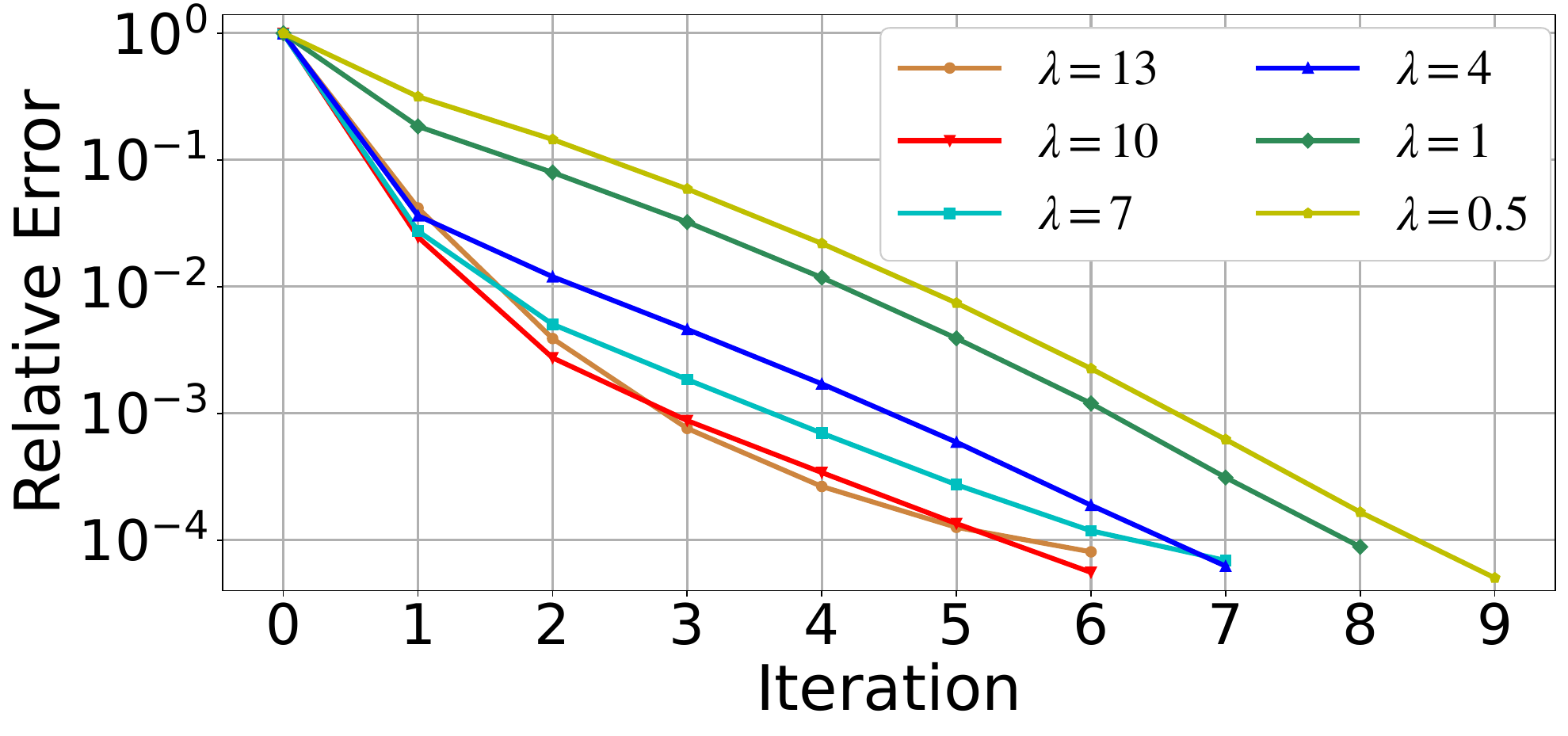} % 图片路径和大小
    \caption{The relative error curves of the iterative results with different $\lambda$}
    \label{fig:cgdifflmd}
\end{figure}

\subsubsection{Experimental Results}

We first verify whether the model can internalize CG’s recursive behavior under the framework. \cref{duallossDiscrepancy} presents a comparative analysis of the intermediate discrepancy between the beginning of the training and after 100k training epochs. It can be observed that even without explicit supervision of the intermediate variables of $\mathbf{r}_{t}$, $\mathbf{d}_{t}$, the intermediates generated by the model still tend to approximate the intermediates of the CG method. This suggests that the model is capable of capturing the structured update logic of CG. 

    \begin{table}[htbp]
    \centering
\begin{tabular}{|c|c|c|}
\hline
    Initial step & 100k epoch & Percentage decrease \\ \hline
     46.1199     & 0.4355          & 99\%                \\ \hline
\end{tabular}
\caption{The discrepancy between the model-generated and the CG intermediate variables before and after training}
\label{duallossDiscrepancy}
\end{table}

In \cref{CG Comparison results of UniFormer}, we show 
the average relative error (over 512 testing examples) 
for each iteration: $\frac{\| \text{NLAF}_ {\text{x}}^{t}-x_{true}\|^2}{\| x_{true}\|^2}$ for both NLAFormer and CG.
To ensure positive definiteness, the input matrices were generated using standard numerical libraries and by adding diagonal matrices whose entries were drawn from log-normal distributions with variances $\sigma = 1$, $1.2$, and $1.4$. This variation in condition numbers facilitates a more comprehensive assessment of the model performance in a variety of numerical scenarios.
As shown in \cref{CG Comparison results of UniFormer}, NLAFormer shows a consistent reduction in the number of iterations required to reach the accuracy threshold compared to CG in multiple cases, even reducing the number of iterations by up to nearly 50\%. 
Crucially, the improvement does not arise from providing explicit additional supervision regarding how to combine operators beyond CG, indicating that the model inherently grasps variations in update dynamics that lead to quicker convergence. This offers empirical backing for the theory that a more effective solver can be crafted through data-driven training. %, our approach presents a framework for exploring numerical linear algebra from a learning-based standpoint, while engaging with well-established algorithmic practices. %This offers empirical backing for the theory that a more effective solver can be crafted through data-driven training, beginning with but eventually exceeding traditional methods. 
% More importantly, the improvement does not arise from providing explicit additional supervision regarding how to combine operators beyond CG, which suggests that the model implicitly learns variations in update behavior that promote faster convergence. These findings provide empirical support for the hypothesis that a more efficient solver can be developed through data-driven training, starting from but ultimately surpassing conventional CG logic. This serves as evidence that learning-based solvers open a new paradigm in numerical computation, where iterative strategies are discovered through data without handcrafted design.

\begin{figure}
	\centering
 \subfigure[$\sigma=1$]{
	\begin{minipage}{0.7\linewidth}
		\centering
		\includegraphics[width=\linewidth]{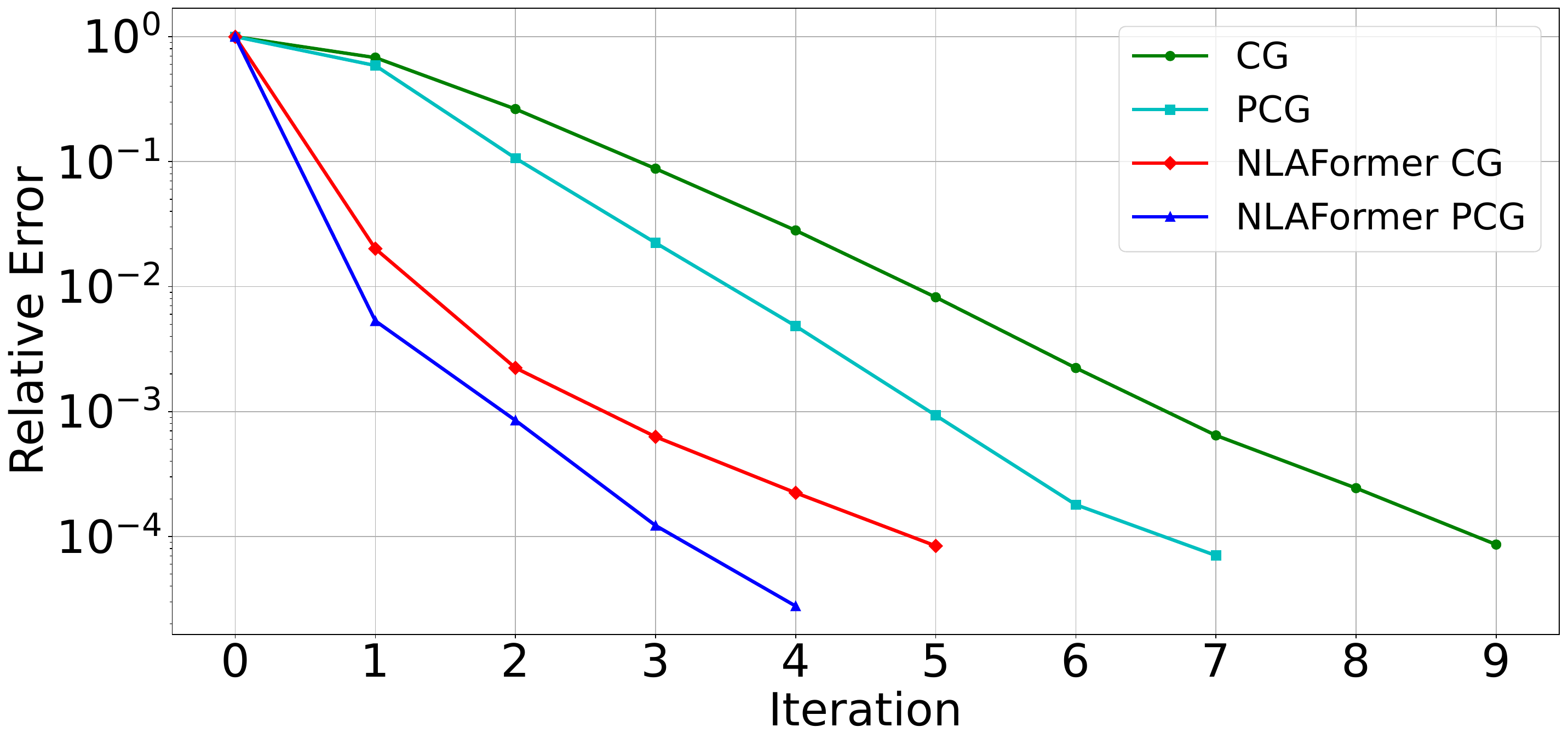}
%\caption{cloth}
%		\label{cloth}
	\end{minipage}}
 \subfigure[$\sigma=1.2$]{
	\begin{minipage}{0.7\linewidth}
		\centering
		\includegraphics[width=\linewidth]{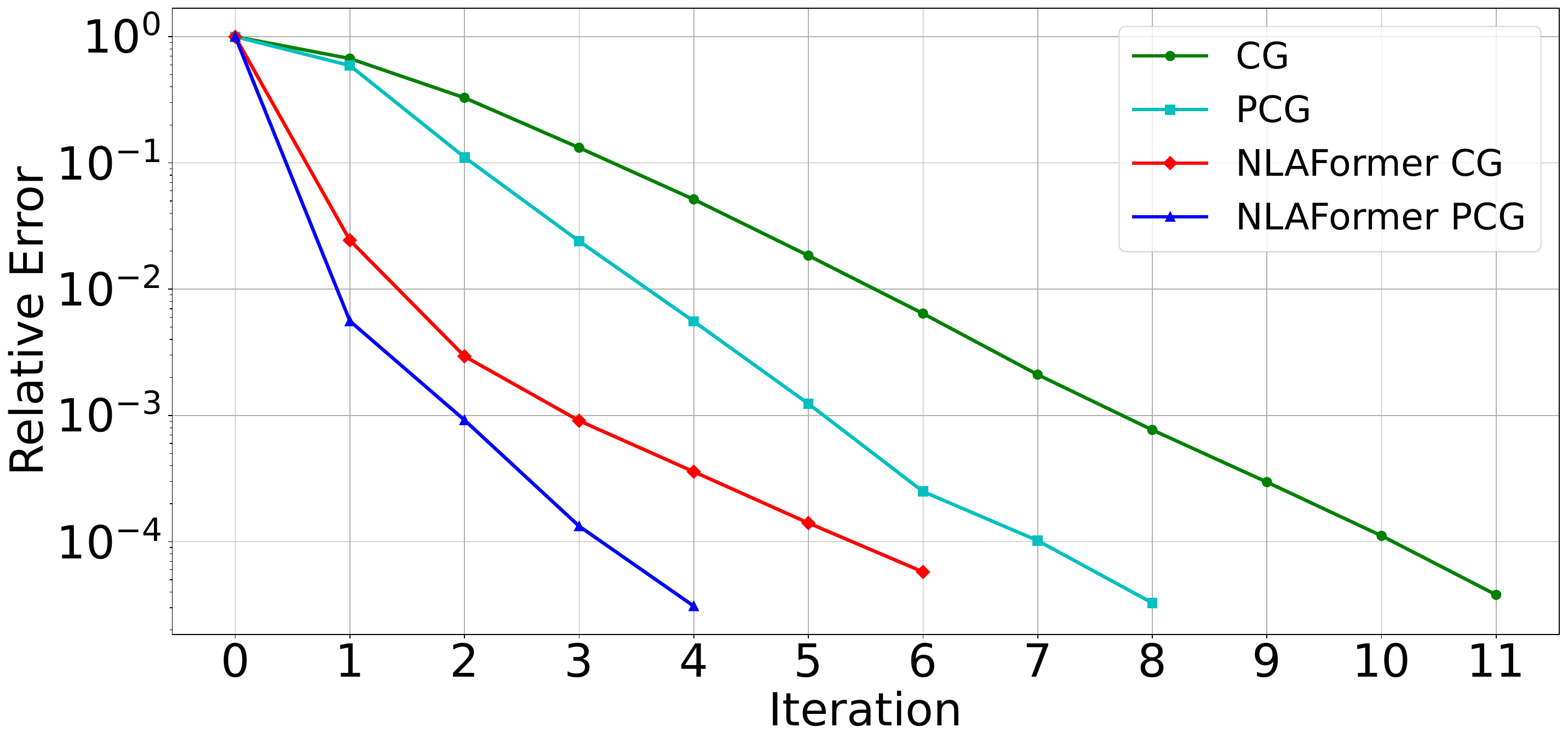}
%\caption{cloth}
%		\label{cloth}
	\end{minipage}}
   \subfigure[$\sigma=1.4$]{
	\begin{minipage}{0.7\linewidth}
		\centering
		\includegraphics[width=\linewidth]{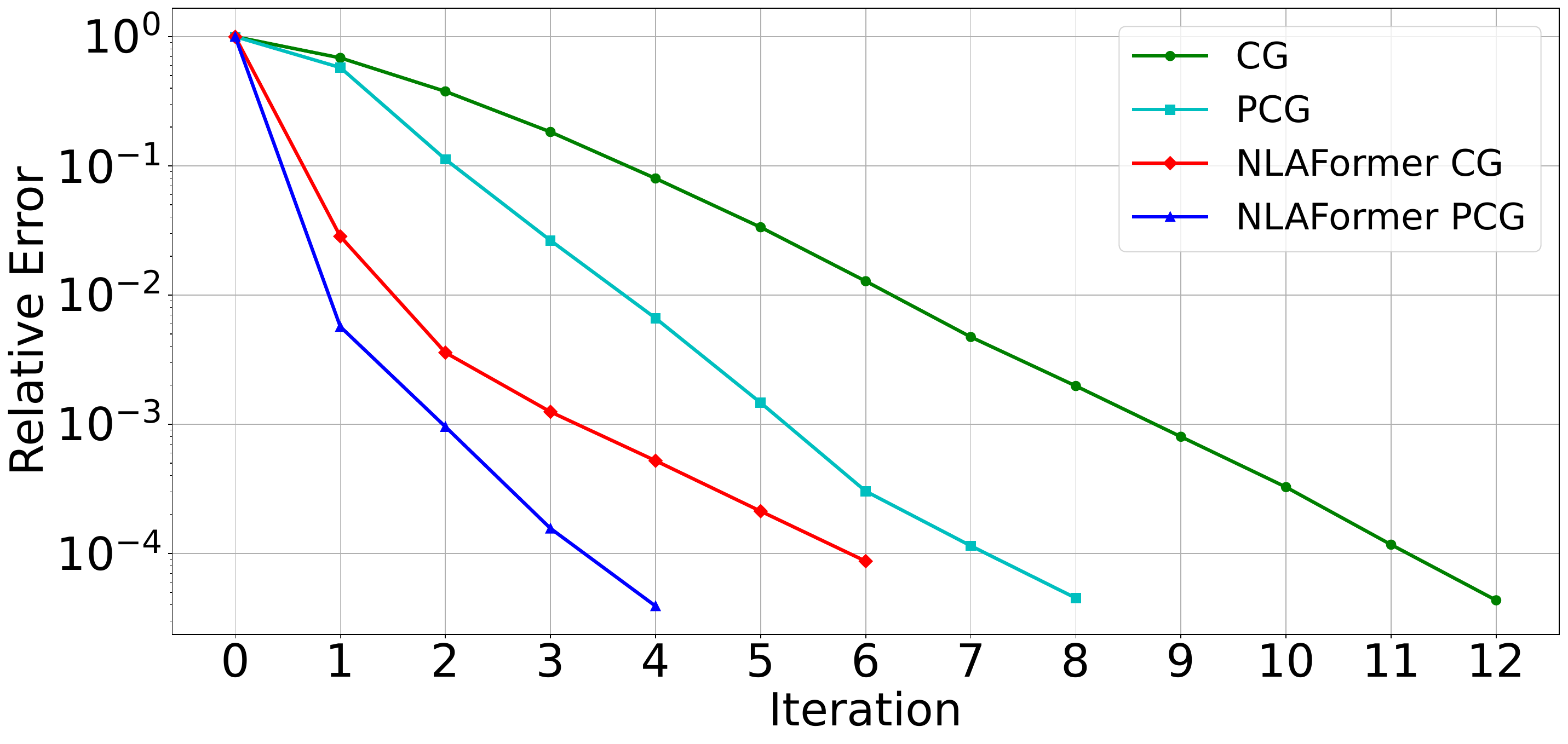}
%\caption{cloth}
%		\label{cloth}
	\end{minipage}}

   \caption{Comparison results of NLAFormer and traditional algorithms on different datasets.}
  \label{CG Comparison results of UniFormer}
\end{figure}

Further, we examine whether NLAFormer can generalize its optimization capability when exposed to a stronger and more sophisticated iterative method, the Jacobi preconditioned conjugate gradient (PCG) algorithm. In this setting, we replace the stepwise supervision signal from CG to PCG without modifying the model structure or providing explicit guidance on how to optimize PCG itself.
The findings presented in \cref{CG Comparison results of UniFormer} demonstrate that NLAFormer, when trained using PCG, achieves a notably faster convergence. In contrast, the classical PCG algorithm requires more than 50\% additional iterations to achieve an equivalent error level. Furthermore, the number of iterations is even reduced by approximately 30\% compared to NLAFormer trained with CG. Its ability to generalize beyond specific baselines and adapt to varying algorithmic patterns highlights its potential as a flexible and learnable framework that not only replicates but also refines classical solvers, providing a foundation for exploring new ways in which data-driven training and classical methods can inform and shape one another.

\subsubsection{Parameter Analysis}

We assess how NLAFormer's convergence behavior is influenced by key training parameters.
\begin{itemize}
\item{Embedding dimension.} As shown in the first row of \cref{CGdiffDIm,PCGdiffDIm}, larger embedding dimensions generally yield better convergence performance, indicating that increased representational capacity enhances the model’s ability to capture iterative dynamics.
\item{Learning window size.} The second row of \cref{CGdiffDIm,PCGdiffDIm} shows that while performance slightly degrades when reducing the learning window (e.g., from 20 to 15), the model maintains competitive convergence behavior. This suggests that NLAFormer does not overfit to a specific window size and retains robustness under reduced supervision length.
\item{Distribution robustness.} Across various matrix distributions of the system (see \cref{CG Comparison results of UniFormer}), NLAFormer consistently maintains a strong performance. This adaptability contrasts with CG, whose iteration count varies significantly depending on matrix conditioning, highlighting the potential of data-driven models to learn representations that generalize across problem variations.
\end{itemize}

\begin{figure}[htb]
	\centering
 \subfigure[$\sigma=1$]{
	\begin{minipage}{0.31\linewidth}
		\centering
		\includegraphics[width=\linewidth]{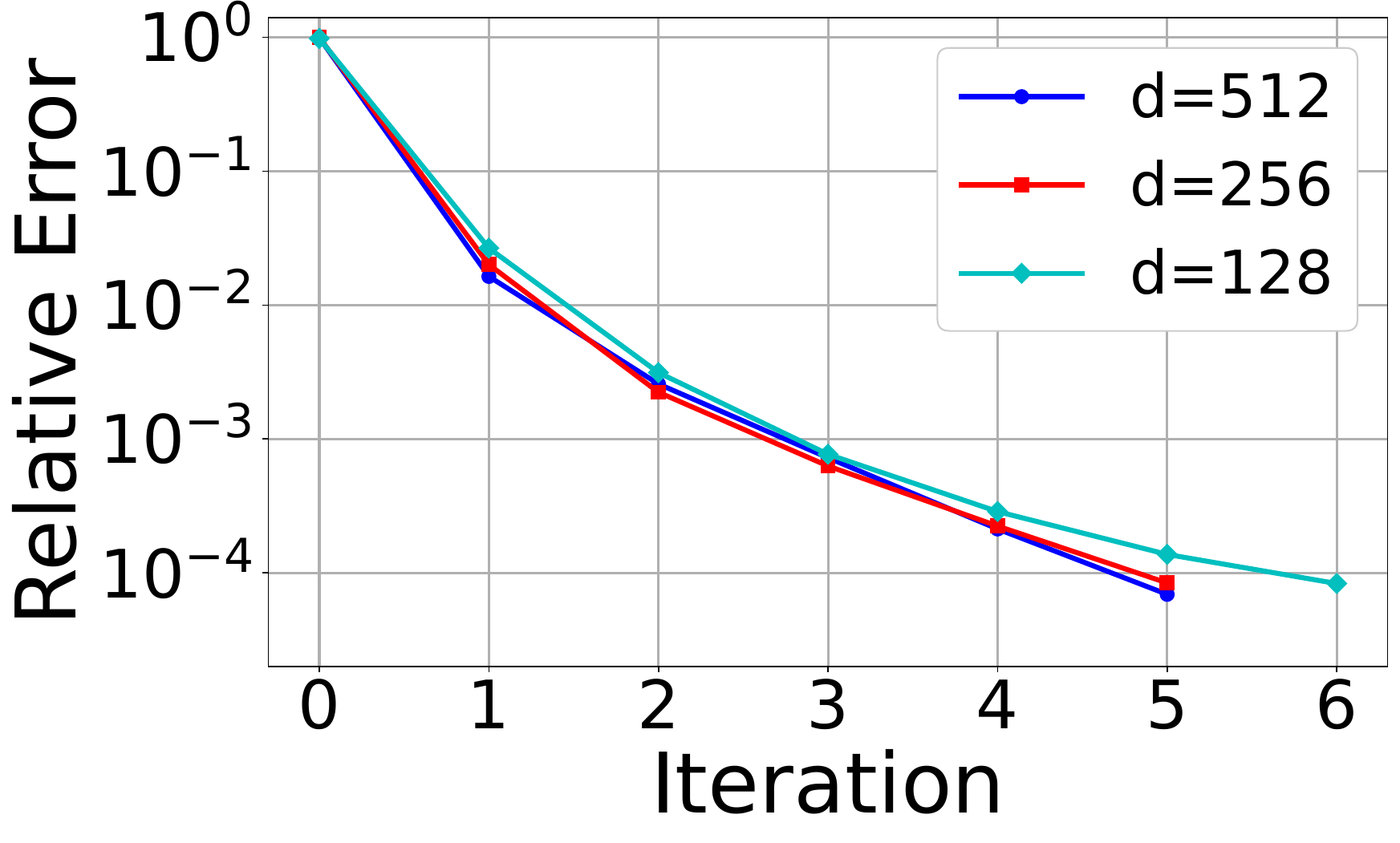}
%\caption{cloth}
%		\label{cloth}
	\end{minipage}}
 \subfigure[$\sigma=1.2$]{
	\begin{minipage}{0.31\linewidth}
		\centering
		\includegraphics[width=\linewidth]{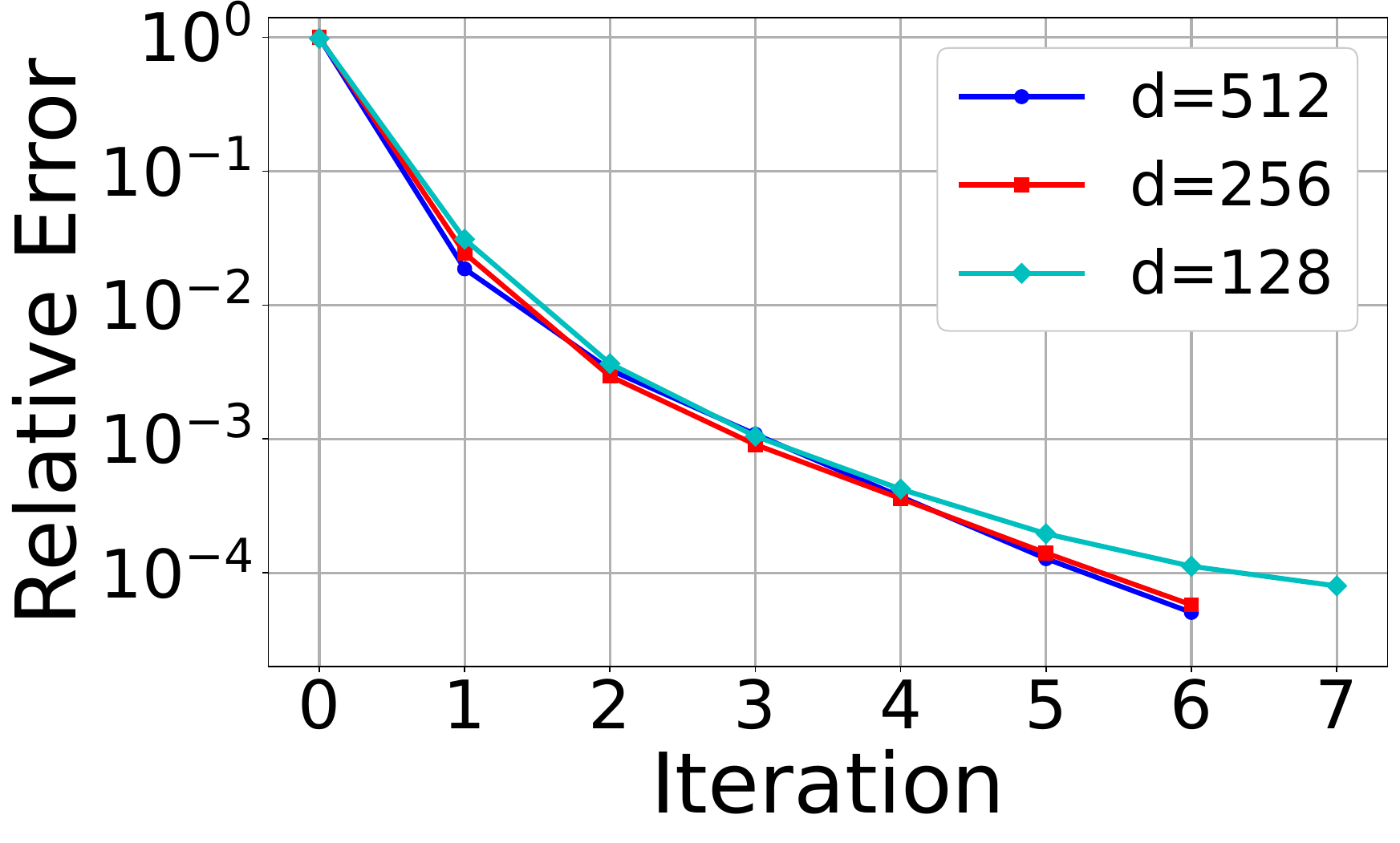}
%\caption{cloth}
%		\label{cloth}
	\end{minipage}}
   \subfigure[$\sigma=1.4$]{
	\begin{minipage}{0.31\linewidth}
		\centering
		\includegraphics[width=\linewidth]{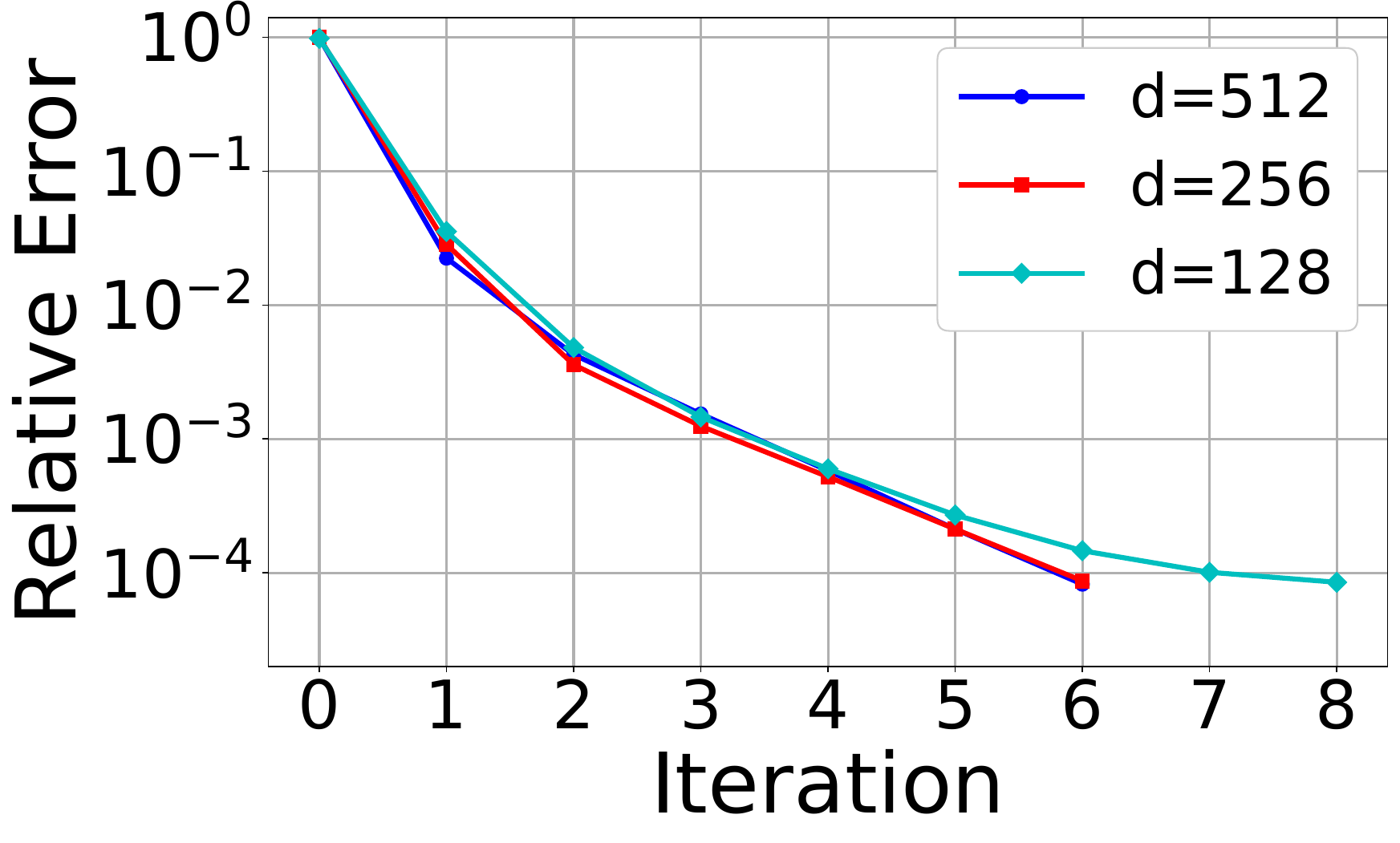}
%\caption{cloth}
%		\label{cloth}
	\end{minipage}}
\subfigure[$\sigma=1$]{
	\begin{minipage}{0.31\linewidth}
		\centering
		\includegraphics[width=\linewidth]{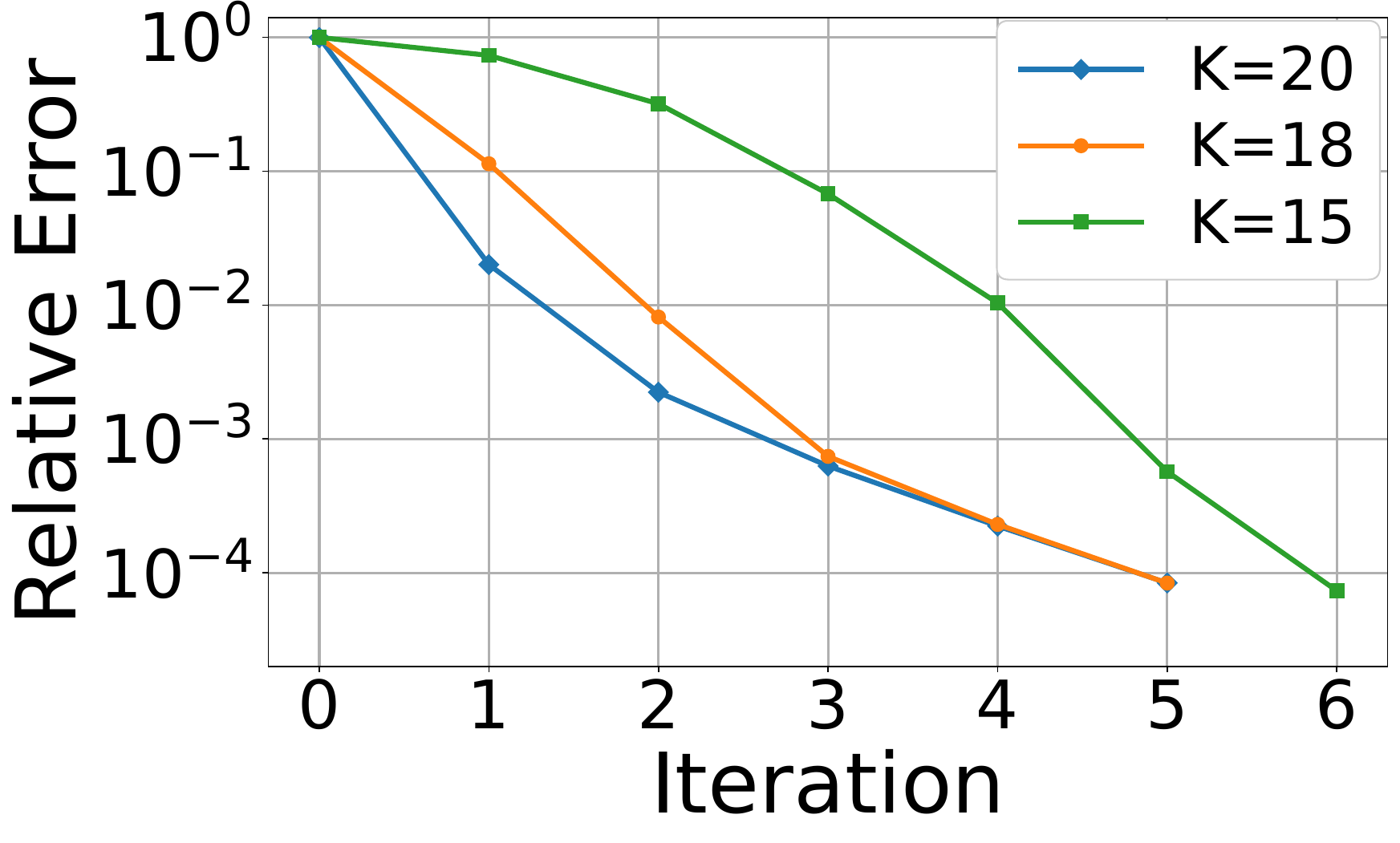}
%\caption{cloth}
%		\label{cloth}
	\end{minipage}}
 \subfigure[$\sigma=1.2$]{
	\begin{minipage}{0.31\linewidth}
		\centering
		\includegraphics[width=\linewidth]{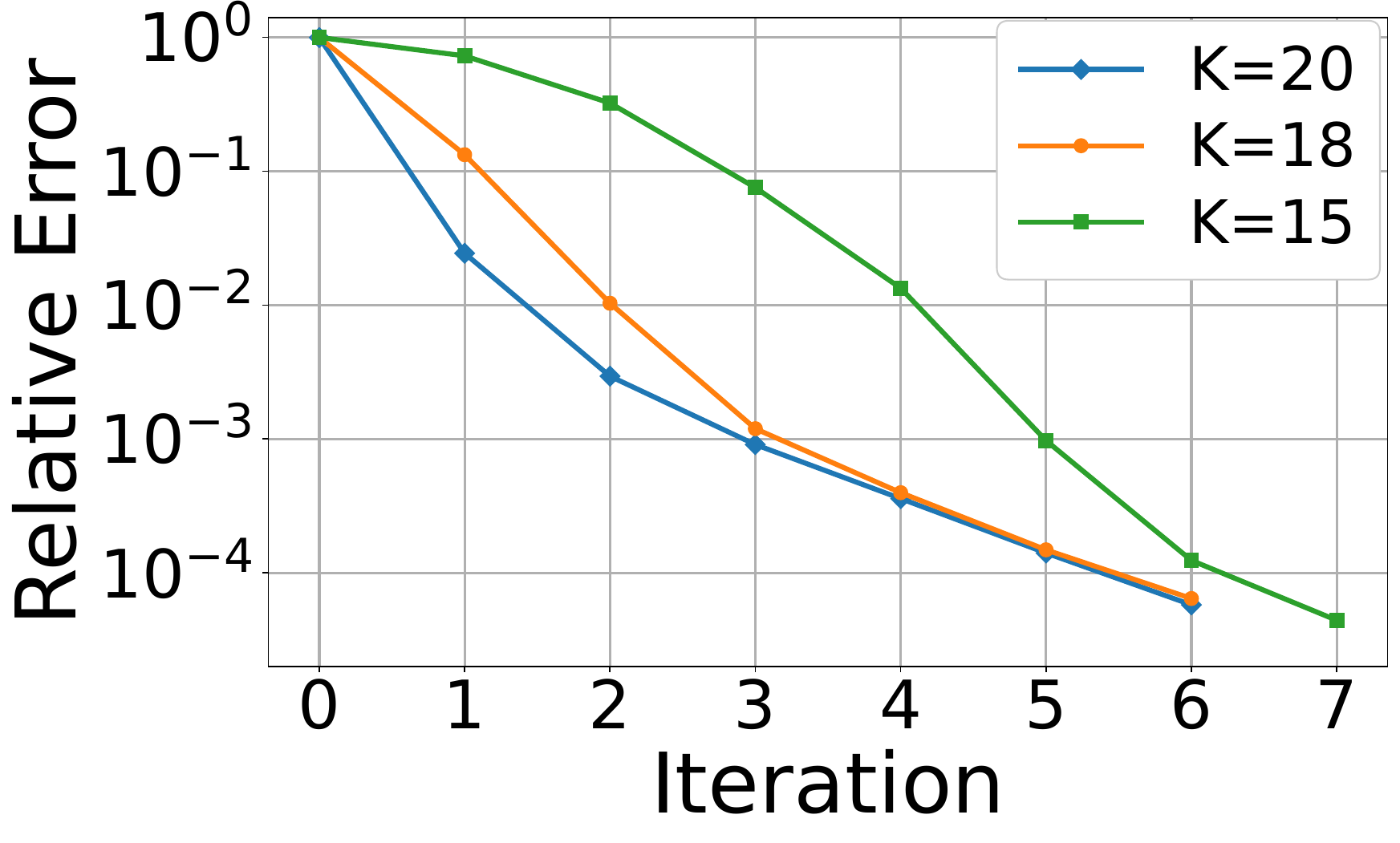}
%\caption{cloth}
%		\label{cloth}
	\end{minipage}}
   \subfigure[$\sigma=1.4$]{
	\begin{minipage}{0.31\linewidth}
		\centering
		\includegraphics[width=\linewidth]{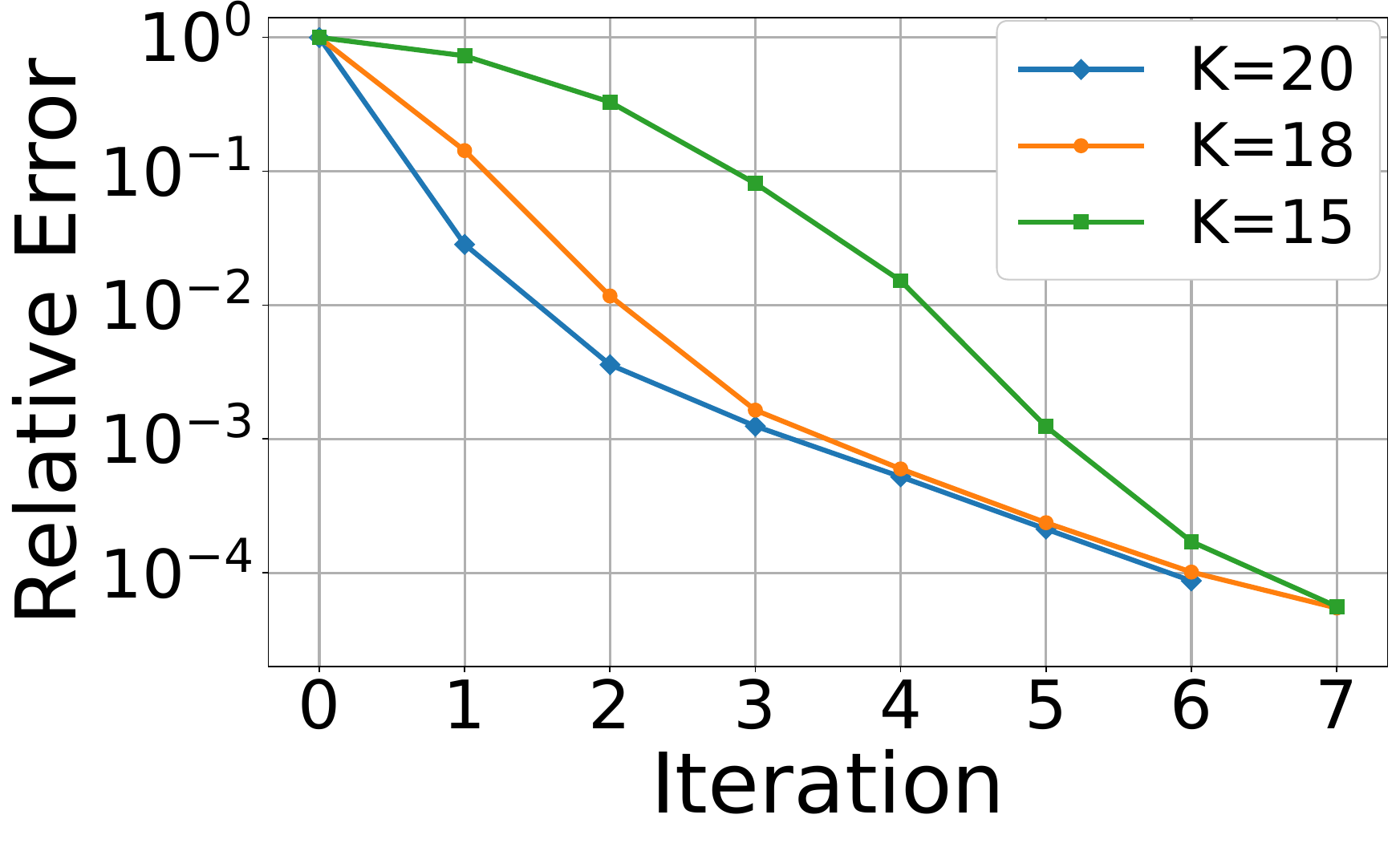}
%\caption{cloth}
%		\label{cloth}
	\end{minipage}}

   \caption{Comparison results across different dimensions and learning windows on CG tasks.}
  \label{CGdiffDIm}
\end{figure}

\begin{figure}[htb]
	\centering
 \subfigure[$\sigma=1$]{
	\begin{minipage}{0.31\linewidth}
		\centering
		\includegraphics[width=\linewidth]{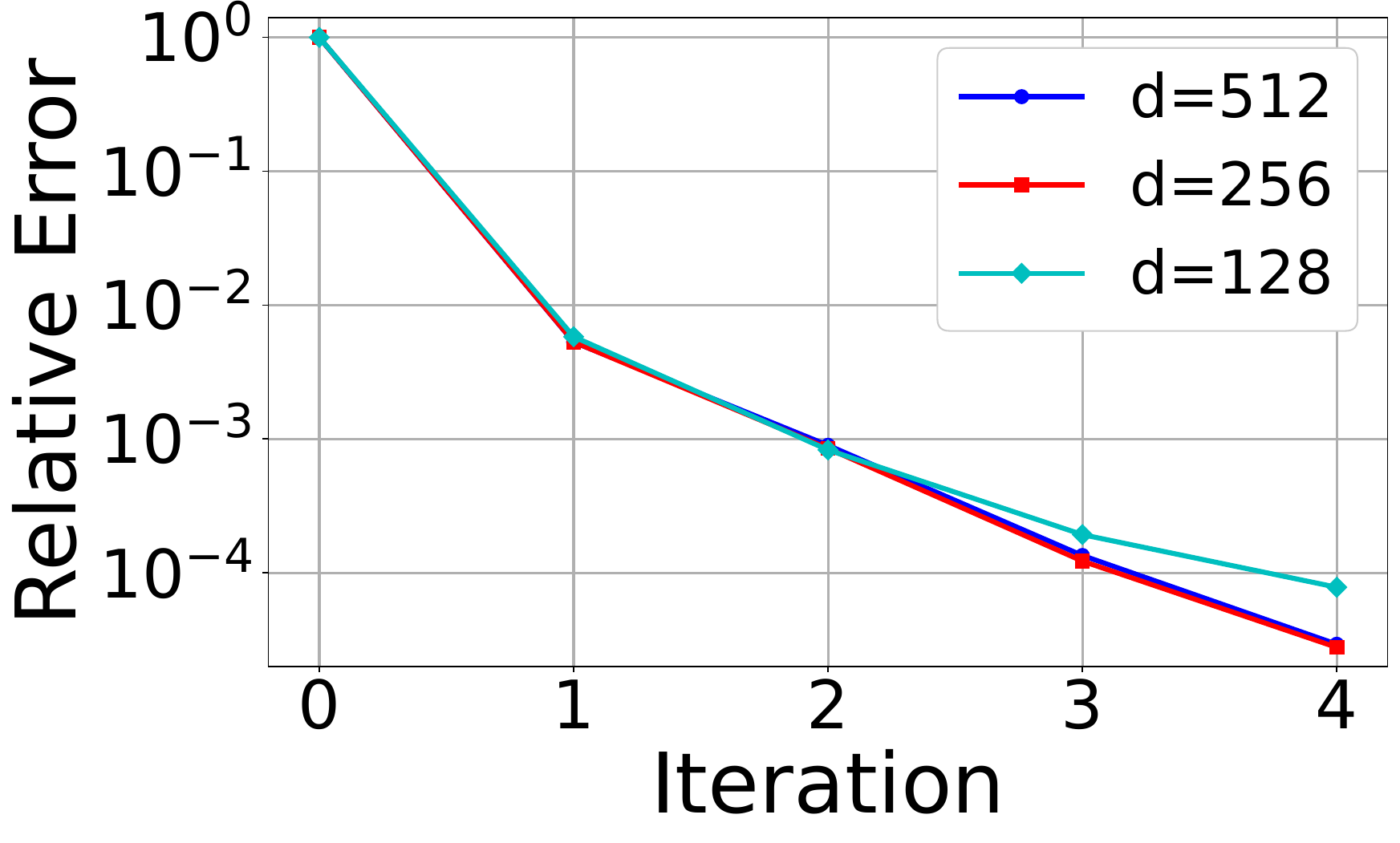}
	\end{minipage}}
 \subfigure[$\sigma=1.2$]{
	\begin{minipage}{0.31\linewidth}
		\centering
		\includegraphics[width=\linewidth]{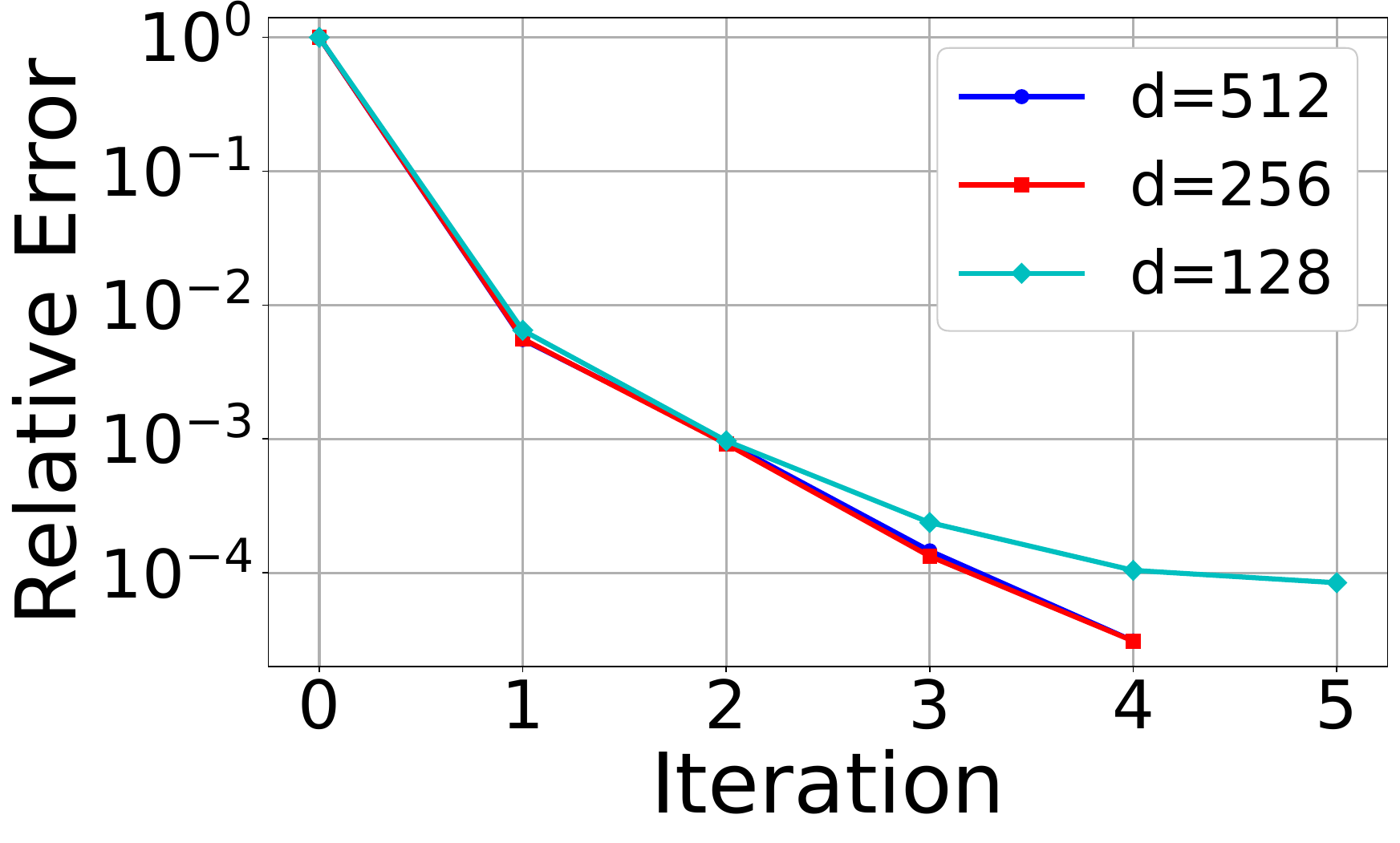}
	\end{minipage}}
   \subfigure[$\sigma=1.4$]{
	\begin{minipage}{0.31\linewidth}
		\centering
		\includegraphics[width=\linewidth]{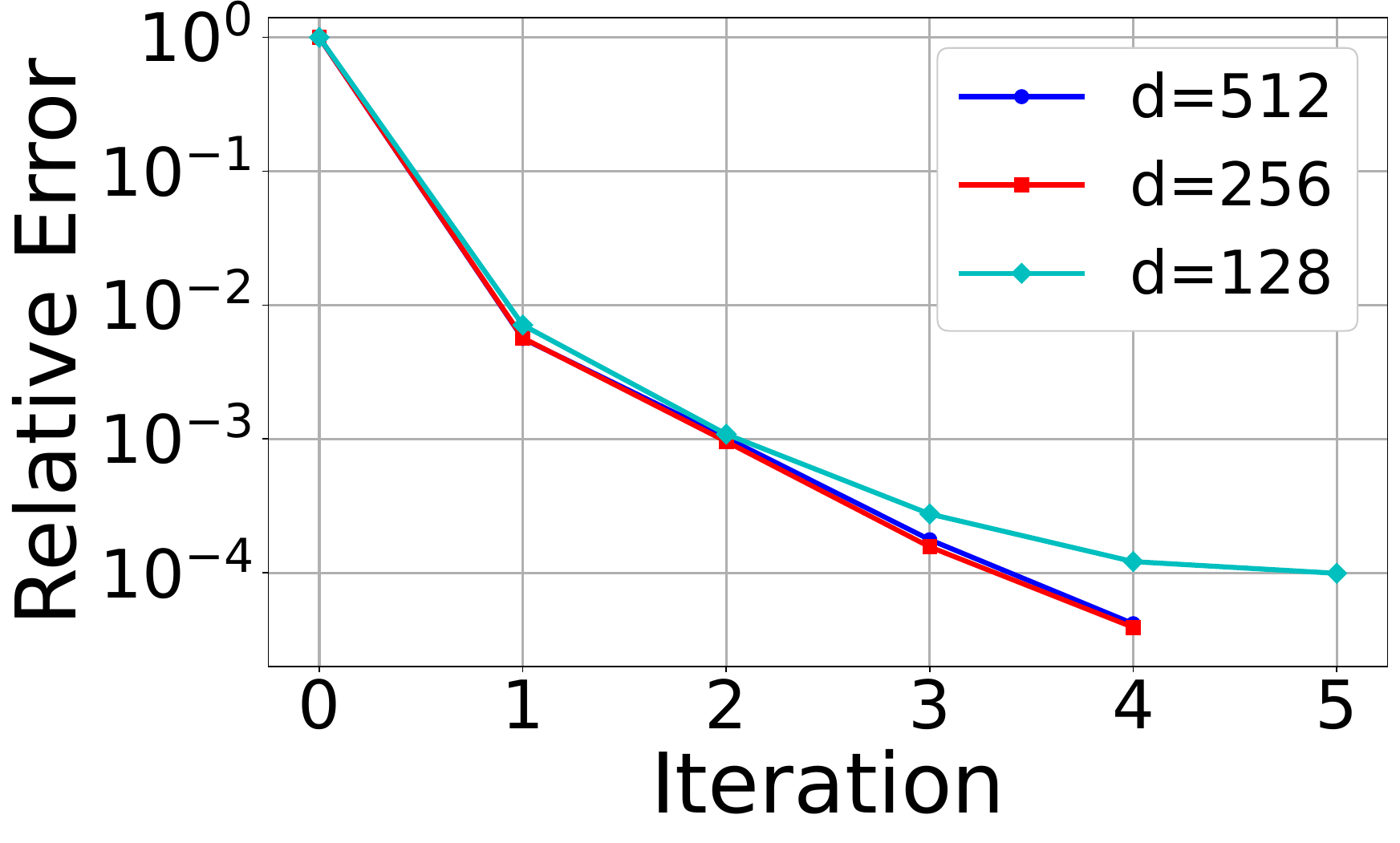}
	\end{minipage}}
    
 \ \subfigure[$\sigma=1$]{
	\begin{minipage}{0.31\linewidth}
		\centering
		\includegraphics[width=\linewidth]{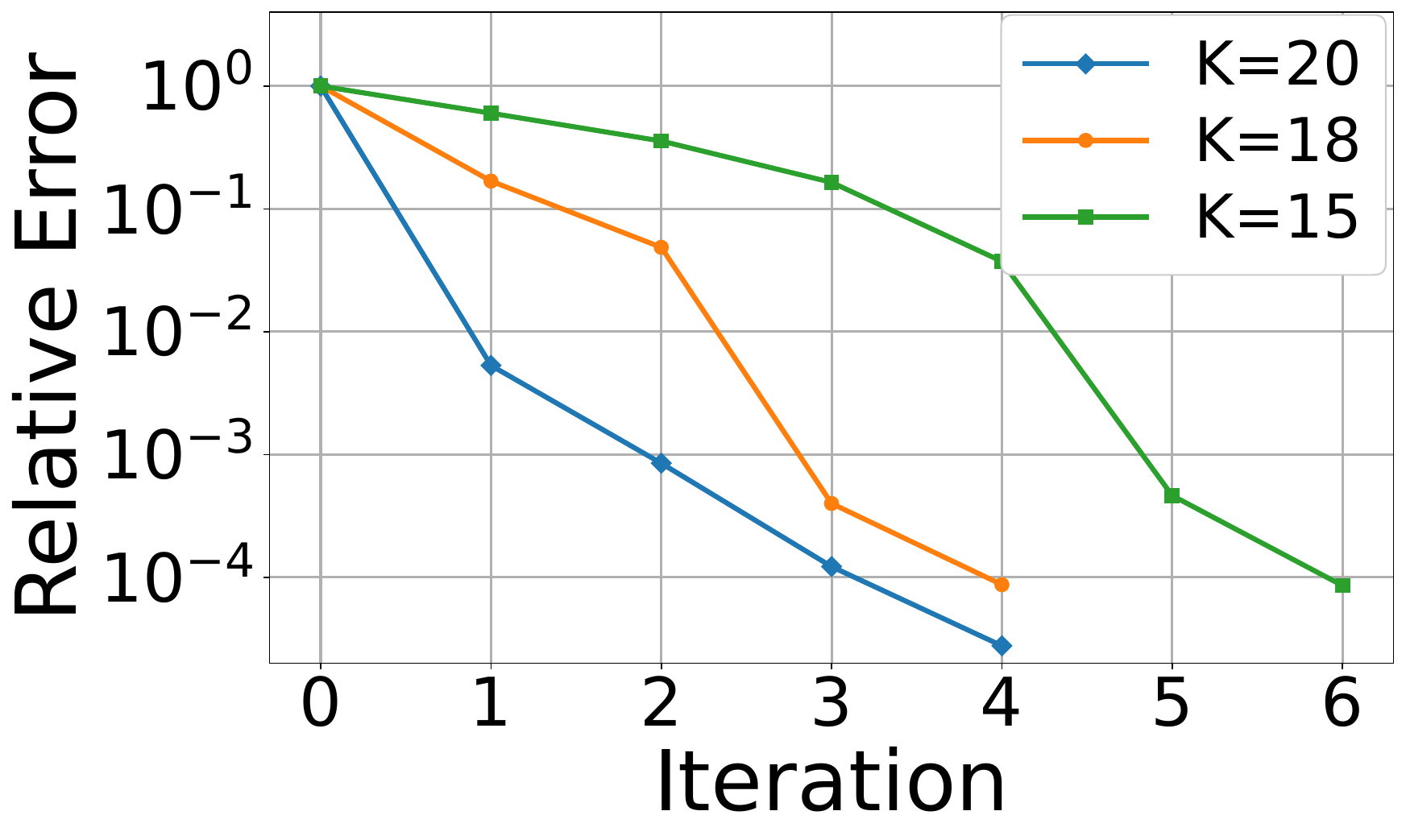}
	\end{minipage}}
 \subfigure[$\sigma=1.2$]{
	\begin{minipage}{0.31\linewidth}
		\centering
		\includegraphics[width=\linewidth]{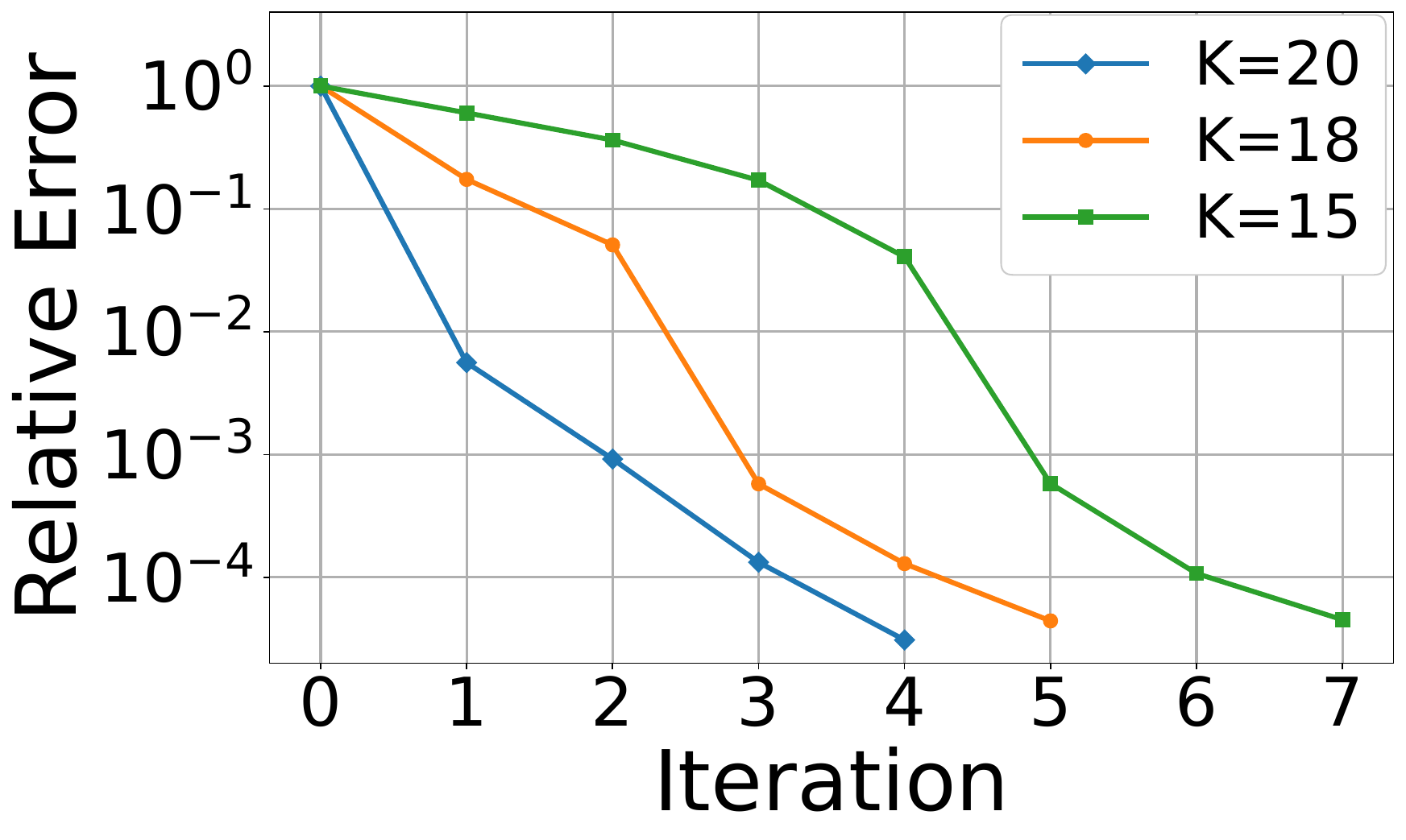}
	\end{minipage}}
   \subfigure[$\sigma=1.4$]{
	\begin{minipage}{0.31\linewidth}
		\centering
		\includegraphics[width=\linewidth]{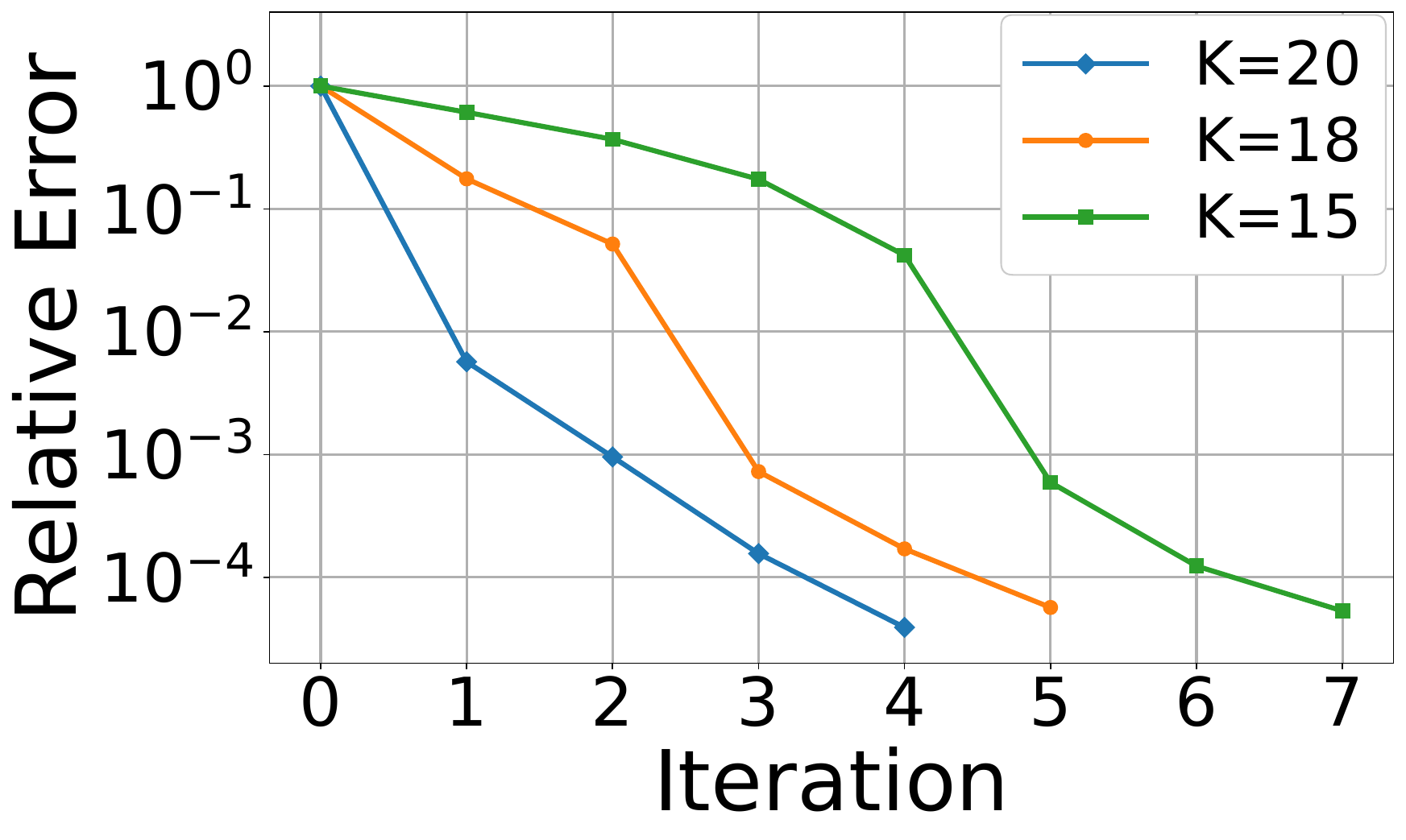}
	\end{minipage}}
   \caption{Comparison results across different dimensions and learning windows on PCG tasks.}
  \label{PCGdiffDIm}
\end{figure}

\section{Concluding Remarks}

We proposed NLAFormer, a transformer-based framework for numerical linear algebra that efficiently represents fundamental operations with far smaller input size and fewer layers than loop-transformers. Using the conjugate gradient method as a case study, we demonstrated its ability to express complex iterative numerical algorithms. Experiments further show that NLAFormer can internalize the iterative logic of classical solvers and, through data-driven training, identify update strategies for faster convergence, underscoring its potential as a compact and versatile approach for combining deep learning with numerical computation.%we illustrated its capability to communicate intricate iterative numerical algorithms. Further experiments indicate that NLAFormer successfully emulates traditional solver behavior with high precision and, through data-driven training, identifies update strategies for faster convergence, underscoring its potential as a compact and versatile approach for combining deep learning with numerical computation.
%NLAFormer introduces an efficient mapping from transformer architectures to the field of numerical linear algebra. Unlike instruction-level simulation in the loop-transformer that relies on complex control logic, NLAFormer implements numerical procedures in a compact and data-efficient manner. This architecture avoids hard-coded logic, enabling it to internalize core operations and emulate recursion efficiently, marking a theoretical advance in using neural models as structured computational frameworks. Beyond this theoretical contribution, NLAFormer signals a possible paradigm shift in numerical linear algebra. NLAFormer learns the behavior of the solver from data. This data-centric approach transforms solver development by reducing the need for hand-crafted algorithms, enabling generalization across tasks, and even surpassing conventional methods through data-driven optimization. %In essence, NLAFormer is not merely a computational tool but a conceptual framework that integrates mathematical structure with the flexibility of modern neural networks. Building upon the current findings of NLAFormer, several research directions can further enhance its methodology.

Looking ahead, several promising directions remain to be explored. Many important numerical computation tasks, such as nonlinear systems and partial differential equations, fall outside the current category. Extending our framework to these areas will test its generality and advance a unified neural approach to numerical analysis.
On the other hand, 
integrating NLAFormer with classical solvers in a complementary fashion could combine the data-driven adaptability of deep models with the rigor of established numerical methods. For example, NLAFormer could rapidly generate high-quality initial guesses, which classical algorithms can then refine for precision and convergence.
Pursuing these directions will strengthen the theoretical foundation and practical applicability of NLAFormer, potentially redefining how numerical computation is conceived and implemented in both academic research and industrial applications.

\bibliographystyle{siamplain}
\bibliography{reference_new}

\end{document}

%% file: ex_shared.tex
\usepackage{lipsum}
\usepackage{amsfonts}
\usepackage{graphicx}
\usepackage{epstopdf}
\usepackage{algorithmic}
\usepackage{multirow}
\usepackage{mathtools}%new add
\usepackage{mathrsfs}%new add
\usepackage{subfigure}%new add
\usepackage[numbers]{natbib}
\usepackage{float}%new add
\ifpdf
  \DeclareGraphicsExtensions{.eps,.pdf,.png,.jpg}
\else
  \DeclareGraphicsExtensions{.eps}
\fi

\newsiamremark{remark}{Remark}
\newsiamremark{hypothesis}{Hypothesis}
\crefname{hypothesis}{Hypothesis}{Hypotheses}
\newsiamthm{claim}{Claim}
\newsiamremark{fact}{Fact}
\crefname{fact}{Fact}{Facts}

\headers{}{Zhantao Ma, Yihang Gao, Michael K. Ng}

\title{NLAFormer: Transformers Learn Numerical Linear Algebra Operations
\thanks{Submitted to the editors August 2025.
This work was partly supported by the National Key Research and Development Program of China under Grant 2024YFE0202900 and Joint NSFC and RGC N-HKU769/21.}}

\author{Zhantao Ma\thanks{Department of Mathematics, The University of Hong Kong, Pokfulam, Hong Kong (mazhantao@connect.hku.hk).}
\and Yihang Gao\thanks{Department of Mathematics, 
National University of Singapore, Singapore (gaoyh@nus.edu.sg).}
\and Michael K. Ng\thanks{Department of Mathematics, Hong Kong Baptist University, Kowloon Tong, Hong Kong (michael-ng@hkbu.edu.hk).}
}

\usepackage{amsopn}